\documentclass[11pt]{amsart}

\usepackage[utf8]{inputenc}
\usepackage[T1]{fontenc}
\usepackage[french,british]{babel}
\usepackage{lmodern}
\usepackage[a4paper]{geometry}
\usepackage{amsmath,amssymb,amsthm}
\usepackage[shortlabels]{enumitem}
\usepackage[all]{xy}

\theoremstyle{plain}
\newtheorem{theo}{Theorem}[section]
\newtheorem{lem}[theo]{Lemma}
\newtheorem{cor}[theo]{Corollary}
\newtheorem{prop}[theo]{Proposition}
\theoremstyle{definition}
\newtheorem{defi}[theo]{Definition}

\newtheorem{rem}[theo]{Remark}

\theoremstyle{plain}

\theoremstyle{definition}

\theoremstyle{plain}
\newtheorem{theorem}{Theorem}

\title[Modular Weil rep. and compatibility of cuspidals with congruences]{Modular Weil representation and compatibility of cuspidals with congruences}
\author{Justin Trias}
\date{}

\begin{document}

\maketitle

\begin{abstract} Let $F$ be a non-archimedean local field of characteristic different from $2$ and of residual characteristic $p$. We generalise the theory of the Weil representation over $F$ with complex coefficients to $\ell$-modular representations
\textit{i.e.} when the complex coefficients are replaced by a coefficient field $R$ of characteristic $\ell \neq p$. We obtain along the way a generalisation of the Stone-von Neumann theorem to the $\ell$-modular setting, together with the Weil representation with coefficients in $R$ on the $R$-metaplectic group. Surprisingly enough, the latter $R$-metaplectic group happens to be split over the symplectic group if $\ell = 2$. The theory also makes sense when $F$ is a finite field of odd characteristic. We also establish the irreducibility of the theta lift in the cuspidal case as long as $\ell$ does not divide the pro-orders of the groups at stake and we provide a compatibility to congruences in this setting via an integral version of the theta lift. \end{abstract}

\tableofcontents

\section*{Introduction}

The theta correspondence plays a key role in the theory of automorphic forms as one of a few explicit methods to deal with representations of two groups forming a dual pair in a symplectic group. It allowed to establish the local Langlands correspondence \cite{gan_takeda_gsp4} for $\textup{GSp}_4$, provided relations between Fourier coefficients of modular forms and special values of $L$-functions via the Shimura--Waldspurger correspondence \cite{waldspurger_shimura,waldspurger_shimura_quaternions} as well as deep relations with the formal degree \cite{gan_ichino}. This correspondence has a local and global version, which are both built via the local and global Weil representation.

On the other hand, there is an interest in considering representations with more general coefficient fields than the complex numbers, as the classical Weil representation is. The study of $\ell$-modular representation of $p$-adic groups -- here $\ell$-modular means over a field of positive characteristic different from $p$ -- was initiated by Vignéras \cite{vigneras_gl_2} and was motivated by conjectures of Serre about congruences between modular forms. It has been an active research topic ever since, expanding towards families of representations as well \textit{i.e.} over coefficient rings putting together characteristic zero fields and positive characteristic fields.

In this perspective, we propose to generalise the construction of the local Weil representation to the $\ell$-modular setting, where local means for us over a non-archimedean local field, so we exclude the archimedean case. We also study some properties of congruences for cuspidal representations and provide an integral version of the theta lift in this case. The paper is divided into three parts, and so is the rest of the introduction when we explore our results in more detail. The first part (Sections \ref{sec:the_heisenberg_rep}-\ref{sec:props_of_met_group_and_weil_rep}) deals with the $\ell$-modular generalisation of the Stone-von Neumann theorem, the Heisenberg representation and its models, the metaplectic group, the Weil representation and its models, and some other classical properties. The second part (Sections \ref{sec:non_norm_Weil_factor}-\ref{sec:the_met_cocycle}) establishes a formula for the metaplectic cocycle similiar to \cite{rao} via an explicit section into the metaplectic group. The third part (Sections \ref{sec:around_modular_theta}-\ref{sec:congruences_for_cuspidal_theta_lifts}) states what an $\ell$-modular local theta correspondence should look like and studies in more detail the cuspidal case, generalising a result of Kudla as long as $\ell$ is large enough. This paper replaces \cite{trias_theta1} and partially improves it thanks to the recent results of \cite{dhkm_conj}.

\subsection{} Let $F$ be a field of characteristic different from $2$, that is either local non-archimedean of residual characteristic $p$ or finite of characteristic $p$. Let $R$ be a field of characteristic $\ell$ and assume there exists a non-trivial smooth character $\psi : F \to R^\times$. In particular, this condition forces $\ell \neq p$. All our characters and representations are assumed to be smooth. Let $W$ be a symplectic space of finite dimension over $F$ and let $H$ be the Heisenberg group. We generalise the theorem of Stone-von Neumann:

\begin{theorem} Let $\psi : F \to R^\times$ be a non-trivial character. Up to isomorphism, there exists a unique irreducible representation $(\rho_\psi,S) \in \textup{Rep}_R(H)$ with central character $\psi$. \end{theorem}

We call the unique isomorphism class of the theorem the Heisenberg representation associated to $\psi$. The Heisenberg representation has explicit models afforded by self-dual subgroups in $W$ such as Lagrangians (the Schr\"odinger model) and self-dual lattices when $F$ is local non-archimedean (the lattice model). Let $(\rho_\psi,S)$ be any model of the Heisenberg representation. The natural action of $\textup{Sp}(W)$ on $H$ preserves the isomorphism class of $(\rho_\psi,S)$, so Schur's lemma gives a projective representation $\sigma_S : \textup{Sp}(W) \to \textup{PGL}_R(S)$. By taking the fibre product $\sigma_S$ lifts to an actual representation $\omega_{\psi,S}$ \textit{i.e.}
$$\xymatrix{
\widetilde{\textup{Sp}}_{\psi,S}^R(W) \ar[d]^{p_S} \ar[r]^{\omega_{\psi,S}} & \textup{GL}_R(S) \ar[d]^{\textsc{red}} \\
\textup{Sp}(W) \ar[r]^{\sigma_S} & \textup{PGL}_R(S) 
}$$
where $\widetilde{\textup{Sp}}_{\psi,S}^R(W) = \textup{Sp}(W) \times_{\textup{PGL}_R(S)} \textup{GL}_R(S)$ is  a central extension of $\textup{Sp}(W)$ by $R^\times$ called the $R$-metaplectic group and $\omega_{\psi,S}$ is called the Weil representation.

The $R$-metaplectic group fits into an exact sequence
$$1 \longrightarrow R^\times \overset{i_S}{\longrightarrow} \widetilde{\textup{Sp}}_{\psi,S}(W) \overset{p_S}{\longrightarrow} \textup{Sp}(W) \longrightarrow 1.$$
We prove that the derived subgroup $\widehat{\textup{Sp}}_{\psi,S}(W)$ of the $R$-metaplectic group is 
\begin{itemize}[label=$\bullet$] 
\item the usual metaplectic group $\textup{Mp}(W)$ if $F$ is local non-archimedean and $\ell \neq 2$;
\item the derived subgroup $[\textup{Sp}(W),\textup{Sp}(W)]$ of $\textup{Sp}(W)$ if $F$ is finite or $\ell = 2$.
\end{itemize}
Note that $[\textup{Sp}(W),\textup{Sp}(W)] = \textup{Sp}(W)$, except in the exceptional case $\textup{Sp}(W) \simeq \textup{SL}_2(\mathbb{F}_3)$. When $F$ is local non-archimedean and $\ell = 2$, the $R$-metaplectic group is actually split since it contains $\textup{Sp}(W)$. In this case, the Weil representation becomes a representation of $\textup{Sp}(W)$ rather than the two-fold cover $\textup{Mp}(W)$. Intuitively, when $\ell = 2$, a genuine representation of $\textup{Mp}(W)$ actually factors through $\textup{Sp}(W)$ since $1 = -1$ in $R$.

We then study classical properties of the $R$-metaplectic group and prove it is a locally profinite group. We also prove other classical facts about the Weil representation, namely it is smooth and admissible. We also give some formulas for the Weil representation in different famous models, such as the Shcr\"odinger and the lattice models.

\subsection{} We then build a section $\textup{Sp}(W) \to \widetilde{\textup{Sp}}_{\psi,S}(W)$ of $p_S$ which is equivalent to giving a map $g \in \textup{Sp}(W) \mapsto M_g \in \textup{GL}_R(S)$ such that $\sigma_S(g) = \textup{RED}(M_g)$ for all $g \in \textup{Sp}(W)$. We do not follow the classical construction of such a section, and propose a new approach, for two reasons.

First of all, the classical constructions, such as the Shale--Segal--Weil representation \cite{rao} or even all usual formulas on the Schr\"odinger model, require to introduce quantities which may not be already in $R$ in the following sense. Let $F$ be $p$-adic and let $R=\mathbb{Q}[\zeta_{p^\infty}]$. Then the theorem of Stone-von Neumann is valid over $\mathbb{Q}[\zeta_{p^\infty}]$, however, the classical sections of $p_S$ are not necessarily defined over $\mathbb{Q}[\zeta_{p^\infty}]$ because they require $\sqrt{q}$ to be in $\mathbb{Q}[\zeta_{p^\infty}]$ where $q$ is the residue cardinality of $F$ \textit{i.e.} we do not necessarily have $\sqrt{q} \in \mathbb{Q}[\zeta_{p^\infty}]$. As a result, in order to define the classical section, one needs to adjoin $\sqrt{q}$ to $\mathbb{Q}[\zeta_{p^\infty}]$, but it seems wrong to do so if we were able to define a section directly valued in $\mathbb{Q}[\zeta_{p^\infty}]$. Note that either $\sqrt{p}$ or $i \sqrt{p}$ belongs to $\mathbb{Q}[\zeta_{p^\infty}]$ when $p \neq 2$, but $i \notin \mathbb{Q}[\zeta_{p^\infty}]$, so only one of them belongs to $\mathbb{Q}[\zeta_{p^\infty}]$.

Secondly, the classical sections make a key use of unitarity in order to normalise some operators and a uniqueness statement such as \cite[Th 3.5]{rao} is impossible to achieve when $R$ has positive characteristic since unitarity does not make sense. These considerations become extremely important if one wants to descend the Weil representation over a number field, as we will do in \cite{trias_rational_weil}.

For these reasons, we adopt a new strategy which is based on a quantity called the non-normalised Weil factor. As opposed to the classical Weil factor, its non-normalised version is guaranteed to remain within the realm of the values of the character $\psi$ and does not require to make any choice in $R$, while the Weil factor requires one to fix $\sqrt{q} \in R$. We also interpret this quantity as a natural way to normalise Fourier transforms, by staying in the range of the values of $\psi$. This allows us to directly define a section
$$\sigma : \textup{Sp}(W) \to \widetilde{\textup{Sp}}_{\psi,S_X}(W)$$
in the Schr\"odinger model $S_X$, which is valued in $\widehat{\textup{Sp}}_{\psi,S_X}(W)$ if $F$ is local non-archimedean and $\ell \neq 2$ and is a group morphism when $F$ is finite or $\ell = 2$. Similarly to \cite{rao}, this section allows us to make explicit the associated $2$-cocyle $\hat{c}$, called the metaplectic cocycle, which is respectively $\{ \pm 1 \}$-valued or trivial. Again, the fact that this cocycle is $\{ \pm 1 \}$-valued or trivial will be central in the Galois descent arguments of \cite{trias_rational_weil}.

\subsection{} We now suppose that $R$ is an algebraically closed field. Let $(H_1,H_2)$ be a dual pair of type I in $\textup{Sp}(W)$. We assume that the inverse images of $H_1$ and $H_2$ in $\textup{Mp}(W)$ are split. In particular, the Weil representation $\omega_\psi^R$ can be pulled back along $H_1 \times H_2 \to \textup{Mp}(W)$. If $\Pi_1 \in \textup{Rep}_R(H_1)$ is irreducible and cuspidal, the largest $\Pi_1$-isotypic quotient of the Weil representation satisfies $(\omega_\psi^R)_{\Pi_1} \simeq \Pi_1 \otimes_R \Theta(\Pi_1)$ where $\Theta(\Pi_1) \in \textup{Rep}_R(H_2)$.

In the classical theta correspondence, \textit{i.e.} when $R=\mathbb{C}$, a famous result \cite{kudla_invent} of Kudla states that, if $\Theta(\Pi_1) \neq 0$ is cuspidal, then it is irreducible. This situation happens for the so-called first occurrence in Witt towers. There is an even stronger version of this result, which says that $\Theta(\Pi_1)$ is either zero or irreducible \cite{mvw}. As $\mathbb{C}$ and $\overline{\mathbb{Q}_\ell}$ are isomorphic, these results are also valid replacing $\mathbb{C}$ by $\overline{\mathbb{Q}_\ell}$.

We assume $\ell$ does not divide the pro-order of $H_1$. Then $\Pi_1$ is an integral representation and for all stable $\overline{\mathbb{Z}_\ell}$-lattices $L$ in $\Pi_1$, we have $L \otimes_{\overline{\mathbb{Z}_\ell}} \overline{\mathbb{F}_\ell} \simeq \pi_1$ where $\pi_1$ is irreducible and cuspidal. By the Brauer-Nesbitt principle, the representation $\pi_1$ does not depend on the choice of $L$. We write $\pi_1 = r_\ell(\Pi_1)$. Therefore
$$(\omega_\psi^{\overline{\mathbb{Q}_\ell}})_{\Pi_1} \simeq \Pi_1 \otimes_{\overline{\mathbb{Q}_\ell}} \Theta(\Pi_1) \textup{ and } (\omega_\psi^{\overline{\mathbb{F}_\ell}})_{\Pi_1} \simeq \pi_1 \otimes_{\overline{\mathbb{F}_\ell}} \Theta(\pi_1)$$
and since $\pi_1=r_\ell(\Pi_1)$, we want to relate $\Theta(\Pi_1) \in \textup{Rep}_{\overline{\mathbb{Q}_\ell}}(H_2)$ and $\Theta(\pi_1) \in \textup{Rep}_{\overline{\mathbb{F}_\ell}}(H_2)$. Recall that $\Theta(\Pi_1)$ is either zero or irreducible.

We are able to do so, \textit{i.e.} we are able to produce congruences, using an integral model of the Weil representation $\omega_\psi^{\overline{\mathbb{Z}_\ell}}$ \cite{trias_theta2}. We show that there exists a decomposition
$$\omega_\psi^{\overline{\mathbb{Z}_\ell}} = e_{\Pi_1} \omega_\psi^{\overline{\mathbb{Z}_\ell}} \oplus (1-e_{\Pi_1})\omega_\psi^{\overline{\mathbb{Z}_\ell}}$$
where $e_{\Pi_1} \omega_\psi^{\overline{\mathbb{Z}_\ell}} \in \textup{Rep}_{\overline{\mathbb{Z}_\ell}}(H_1 \times H_2)$ is a $\overline{\mathbb{Z}_\ell}$-lattice which controls these congruences \textit{i.e.}
$$e_{\Pi_1} \omega_\psi^{\overline{\mathbb{Z}_\ell}} \otimes_{\overline{\mathbb{Z}_\ell}} \overline{\mathbb{Q}_\ell} \simeq \Pi_1 \otimes_{\overline{\mathbb{Q}_\ell}} \Theta(\Pi_1) \textup{ and } e_{\Pi_1} \omega_\psi^{\overline{\mathbb{Z}_\ell}} \otimes_{\overline{\mathbb{Z}_\ell}} \overline{\mathbb{F}_\ell} \simeq \pi_1 \otimes_{\overline{\mathbb{F}_\ell}} \Theta(\pi_1).$$
In particular $\Theta(\pi_1)$ has finite length. This allows us to generalise Kudla's result:

\begin{theorem} Suppose $\ell$ does not divide the pro-orders of $H_1$ and $H_2$. We assume that $\Theta(\Pi_1)$ is the first occurrence index of $\Pi_1$ in the local theta correspondence \textit{i.e.} $\Theta(\Pi_1)$ is irreducible cuspidal. Then $\Theta(\pi_1) = r_\ell(\Theta(\Pi_1))$ is irreducible cuspidal. \end{theorem}

\addtocontents{toc}{\protect\setcounter{tocdepth}{0}}

\section*{Acknowledgements} 

I would like to thank Alberto M\'inguez and Shaun Stevens for their constant support and encouragement. I am also indebted to Wee Teck Gan and Nadir Matringe for their thorough refereeing on my PhD thesis, which serves as the main material for this paper. I also benefited from fruitful discussions with Anne-Marie Aubert, Gianmarco Chinello, Jean-François Dat, Guy Henniart, Vincent Sécherre, Marie-France Vignéras and Jean-Loup Waldspurger.

The author was partially supported by the EPSRC Grant EP/V061739/1. This work was partly funded by the European Union ERC Consolidator Grant, RELANTRA, project number 101044930. Views and opinions expressed are however those of the author only and do not necessarily reflect those of the European Union or the European Research Council. Neither the European Union nor the granting authority can be held responsible for them. 

\addtocontents{toc}{\protect\setcounter{tocdepth}{1}} 

\section*{Notations}

Let $F$ be a field of characteristic different from $2$, that is either a finite field of cardinality $q$ or a non-archimedean local field of residue characteristic $q$. We write $q=p^f$. When $F$ is local non-archimedean, we let $\mathcal{O}_F$ be its ring of integers and $k_F$ its residue field and we fix a uniformiser $\varpi_F$ in $\mathcal{O}_F$. Let $( \ , \ )_F$ be the quadratic Hilbert symbol, which is trivial if $F$ is finite. If $F$ is local non-archimedean and $V$ is a finite dimensional $F$-vector space, a lattice in $V$ is a free $\mathcal{O}_F$-module of rank the dimension of $V$.

Let $(W, \langle \ , \ \rangle)$ be a symplectic vector space of dimension $n=2m$ over $F$. A subspace $X \subseteq W$ is totally isotropic if $\langle \ , \ \rangle|_{X \times X}$ is identically zero. A totally isotropic subspace is maximal if and only if it has dimension $m$. Such a maximal space is called a Lagrangian in $W$. A complete polarisation $W=X \oplus Y$ is made of two transverse Lagrangians $X$ and $Y$ in $W$. The symplectic group $\textup{Sp}(W)$ is the group of isometries of $W$.

Let $G$ be a locally profinite group \textit{i.e.} a locally compact totally disconnected topological group. Let $K$ be a compact open subgroup of $G$. The pro-order of $K$ is the least common multiple of the cardinality of the finite quotients of $K$ \cite[I.1.5]{vig}. The pro-order $|G|$ of $G$ is the least common multiple of the $|K|$'s where $K$ runs over all compact open subgroups of $G$. When $G$ is a reductive group over $F$, \textit{i.e.} the $F$-points of a reductive algebraic group defined over $F$, we usually have $|G| =  n_f p^k$ where $n_f \in \mathbb{N}$ is prime-to-$p$ and $k \in \mathbb{N} \cup \{\infty\}$.

Let $R$ be a field of characteristic $\ell$. Let $C_c^\infty(G,R)$ be the space of locally constant compactly supported functions on $G$ valued in $R$. If $G$ contains a compact open subgroup of invertible pro-order in $R$, there exists a Haar measure $\mu$ of $G$ with values in $R$ by \cite[I.2.4]{vig} and all such measures are unique up to a scalar in $R^\times$. If a compact open subgroup $K$ has invertible pro-order in $R$, there exists a unique measure $\mu_K$ such that $K$ has volume $1$. We call it the normalised measure on $K$. After fixing a measure of $G$, we can endow $C_c^\infty(G,R)$ with a structure of $R$-algebra and we denote this algebra by $\mathcal{H}_R(G)$ and call it the Hecke algebra.

An $R[G]$-module $V$ is smooth if $\textup{Stab}_G(v) = \{ g \in G \ | \ g \cdot v  = v \}$ is open in $G$ for all $v \in V$. We denote by $\textup{Rep}_R(G)$ the category of smooth $R[G]$-modules, also called smooth representations of $G$. In an arbitrary $R[G]$-module $V$, one can consider the smooth vectors $V^\infty$ \textit{i.e.} the subspace of $v \in V$ such that $\textup{Stab}_G(v)$ is open. Then $V \mapsto V^\infty$ is a functor, which is left-exact but not right-exact. For $V \in \textup{Rep}_R(G)$, we define its contragredient $V^\vee \in \textup{Rep}_R(G)$ by $V^\vee = \textup{Hom}_R(V,R)^\infty$. Let $H$ be a closed subgroup of $G$, we define a functor $\textup{Ind}_H^G : \textup{Rep}_R(H) \to \textup{Rep}_R(G)$ where for $\sigma \in \textup{Rep}_R(H)$, we associate the space $\textup{Ind}_H^G(\sigma)$ of functions $f : G \to \sigma$ such that $f(hg) = \sigma(h) f(g)$ and $f$ is smooth, endowed with the  smooth $G$-action $g \cdot f(g') = f(g' g)$. We also define the subfunctor $\textup{ind}_H^G$ of $\textup{Ind}_H^G$ by moreover requiring that $f$ has compact support modulo $H$.

For $n \in \mathbb{N}$, we denote by $\zeta_n \in \mathbb{C}$ the usual primitive $n$-root of unity \textit{i.e.} $\zeta_n = e^{\frac{2 i \pi}{n}}$. If there exists a non-trivial smooth (additive) character $\psi : F \to R^\times$, then necessarily the characteristic $\ell$ of $R$ is different from $p$. Moreover $R$ must contain enough $p$-roots or $p$-power roots of unity. Let $\mathbb{Z}[\zeta_{p^\infty}] = \cup_k \mathbb{Z}[\zeta_{p^k}]$ and let
$$\mathcal{A} = \left\{ \begin{array}{cc}
\mathbb{Z}[\frac{1}{p},\zeta_{p^\infty}] & \textup{ if } \textup{char} (F) = 0; \\
 & \\
\mathbb{Z}[\frac{1}{p},\zeta_p] & \textup{ if } \textup{char}(F) > 0.
\end{array} \right.$$
Then there exists a non-trivial character $\psi : F \to R^\times$ if and only if $R$ can be endowed with a structure of $\mathcal{A}$-algebra. We always assume $R$ satisfies this condition.

\section{The Heisenberg representation} \label{sec:the_heisenberg_rep}

The Heisenberg group $H(W,\langle \ , \ \rangle )$, or simply $H$, is the set $W \times F$ endowed with the product topology and the group law
$$(w,t) \cdot (w',t') = \bigg(w+w',t+t'+\frac{\langle w,w' \rangle }{2}\bigg).$$
We identify $F$ with the centre of $H$ via the isomorphism of topological groups $t \mapsto (0,t)$ and identify $W$ with the subset $W \times \{ 0 \}$ of $H$ via the homeomorphism $\delta : w \mapsto (w,0)$.

\subsection{} We generalise the Stone-von Neumann theorem \cite[Chap 2, Th I.2]{mvw} to the modular setting. Recall that $R$ is a coefficient field such that there exists a non-trivial smooth character $\psi : F \to R^\times$, so the characteristic $\ell$ of $R$ has to be different from $p$.

\begin{theo} \label{thm:modular_stone_von_neumann} Let $\psi : F \to R^\times$ be a non-trivial character. Up to isomorphism, there exists a unique irreducible representation $(\rho_\psi,S) \in \textup{Rep}_R(H)$ with central character $\psi$. \end{theo}

\begin{proof} The proof is the same as \cite[Chap 2, Th I.2]{mvw}. We recall its main ingredients. A first candidate is $\textup{ind}_F^H(\psi)$ which has the right central character, but it fails to be irreducible. We can construct larger subgroups than $F$ so that the induced representation is irreducible, as we now explain. Let $A$ be a closed subgroup of $W$ and define its orthogonal as $A^\perp = \{ w \in W \ | \ \forall a \in A, \ \psi(\langle w,a \rangle )=1 \}$. The arguments in \cite[Chap 2, I.3]{mvw} are still valid in the modular setting, so we obtain:

\begin{lem} \label{lem:stonve_von_neumann_existence} We suppose that $A$ is self-dual \textit{i.e.} $A=A^\perp$. Then:
\begin{enumerate}[label=\textup{\alph*)}]
\item there exists a smooth character $\psi_A$ of the subgroup $A_H=A \times F$ of $H$ whose restriction to $F$ is $\psi$ ;
\item for all smooth characters $\psi_A$ of $A_H$ extending $\psi$, the compact induction $\textup{ind}_{A_H}^{H}(\psi_A)$ is irreducible. \end{enumerate} \end{lem}

As a result of this lemma, there exist irreducible representations $(\rho_\psi,S)$ of $H$ with central character $\psi$. They enjoy very explicit models thanks to these self-dual subgroups, such as Lagrangians and self-dual lattices.

To show uniqueness, we follow \cite[Chap 2, Lem I.5 \& I.6]{mvw} whose proofs rely on the inversion formula for the Fourier transform, which is still valid in the modular setting \cite[I.3.10]{vig}. This leads to:

\begin{lem} \label{lem:stonve_von_neumann_uniqueness} Let $(\rho_\psi,S) \in \textup{Rep}_R(H)$ be irreducible with central character $\psi$. Then
$$\rho_\psi^\vee \otimes_R \rho_\psi \simeq \textup{ind}_{F}^H(\psi) \textup{ in } \textup{Rep}_R(H \times H).$$ \end{lem}

\noindent The uniqueness of $(\rho_\psi,S)$ is a consequence of this lemma as $\textup{ind}_F^H(\psi)$ is $\rho_\psi$-isotypic. \end{proof}

\subsection{} For $\psi$ a non-trivial character, we call the unique isomorphism class $(\rho_\psi,S)$ above the Heisenberg representation associated to $\psi$. By extension, any representation in this unique class is also called the Heisenberg representation associated to $\psi$.

The representations $S_A = \textup{ind}_{A_H}^H(\psi_A)$ in Lemma \ref{lem:stonve_von_neumann_existence} provide explicit models of the Heisenberg representation associated to $\psi$. They are particularly important when $A$ is a Lagrangian (Schr\"odinger model) or $A$ is a self-dual lattice (lattice models when $F$ is local non-archimedean).

\subsubsection{Schr\"odinger model} Let $W=X \oplus Y$ be a complete polarisation. Then $X$ is a self-dual subgroup of $W$. Let $S_X = \textup{ind}_{X_H}^H(\psi_X)$ where $\psi_X(x,t) = \psi(t)$ is a character of $X_H = X \times F$. The restriction to $Y$ induces an isomorphism $S_X \simeq C_c^\infty(Y)$ where the action on the right-hand side is given for $h=(w_X+w_Y,t) \in H$ and $f \in C_c^\infty(Y)$ by
$$\rho_\psi(h) f : y \in Y \mapsto \psi \bigg( \langle y,w_X \rangle  + \frac{1}{2} \langle w_Y,w_X \rangle  + t \bigg) f(y+ w_Y) \in R.$$

\subsubsection{Lattice model} Let $F$ be local non-archimedean. Let $A \subseteq W$ be a self-dual lattice. Such lattices always exist according to \cite[Chap 2, I.4 (2)]{mvw}. By Lemma \ref{lem:stonve_von_neumann_existence}, we can extend $\psi$ as a character $\psi_A$ of $A_H = H \times F$ and set $S_A = \textup{ind}_{A_H}^H(\psi_A)$. The restriction to $W \subseteq H$ induces an isomorphism between $S_A$ and the functions in $C_c^\infty(W)$ satisfying
$$f(a+w) = \psi_A(\langle w,a \rangle )f(w) \text{ , for all } a \in A \text{ and } w \in W.$$
The action of $h=(w,t) \in H$ on $f \in C_c^\infty(W)$ as above is given by
$$ \rho_\psi(h)  f : w' \mapsto  \psi(t) \psi \bigg( \frac{1}{2}\langle w',w \rangle \bigg) f(w'+w).$$

\subsection{} We generalise to the modular setting a few classical properties -- which appear in \cite[Chap 2, I.6 \& I.8]{mvw} -- of the complex Heisenberg representation. If $A$ is a self-dual lattice in $W$, we set $S_A = \textup{ind}_{A_H}^H(\psi_A)$ where $\psi_A$ extends $\psi$.

\begin{prop} \label{prop:heisenberg_representation} Let $\rho_\psi$ be the Heisenberg representation associated to $\psi$.
\begin{enumerate}[label=\textup{\alph*)}]
\item The representations $\rho_\psi^\vee$ and $\rho_{\psi^{-1}}$ are isomorphic.
\item The representation $\rho_\psi$ is admissible, absolutely irreducible and satisfies Schur's lemma \textit{i.e.} $\textup{End}_{R[H]}(\rho_\psi) = R$.
\item Any smooth representation of $H$ with central character $\psi$ is semi-simple.
\item Let: 
\begin{itemize}[label=$\bullet$]
\item $(W_1,\langle \ , \ \rangle_1)$ and $(W_2,\langle \ , \ \rangle_2)$ be two symplectic spaces over $F$;
\item $W= W_1 \oplus W_2$ their orthogonal sum;
\item $H(W_1,\langle \ , \ \rangle_1)$ and $H(W_2,\langle \ , \ \rangle_2)$ the associated Heisenberg groups;
\item $\rho_\psi^1$ and $\rho_\psi^2$ the respective Heisenberg representations associated to $\psi$.
\end{itemize}
Then the representation $\rho_\psi^1 \otimes_R \rho_\psi^2 \in \textup{Rep}_R(H(W,\langle \ , \ \rangle ))$ can be identified with the Heisenberg representation associated to $\psi$ in the following model:
$$(w_1+w_2,t) \mapsto \psi(t) \times \big( \rho_\psi^1((w_1,0)) \otimes \rho_\psi^2((w_2,0)) \big).$$
\end{enumerate} \end{prop}

\begin{proof} a) Both are irreducible and have central character $\psi^{-1}$, we apply Theorem \ref{thm:modular_stone_von_neumann}.

\noindent b) Because $\rho_\psi^\vee \simeq \rho_{\psi^{-1}}$, we also have $(\rho_\psi^\vee)^\vee \simeq \rho_\psi$. Therefore $\rho_\psi$ is reflexive, so $\rho_\psi$ has to be admissible. By compatibility of the induction with scalar extension, the representation $S_A \otimes_R R' \in \textup{Rep}_{R'}(H)$ is the Heisenberg representation associated to $\psi' : F \to (R')^\times$ obtained by composition. This implies that $\rho_\psi$ is absolutely irreducible by taking $R'=\overline{R}$ an algebraic closure of $R$. The representation $\rho_\psi$ is admissible and absolutely irreducible, it satisfies Schur's lemma by \cite[I.6.9]{vig}.

\noindent c) From Lemma \ref{lem:stonve_von_neumann_uniqueness}, the representation $\textup{ind}_F^H(\psi)$ is $\rho_\psi$-isotypic. In the category of representations with central character $\psi$, the latter $\textup{ind}_F^H(\psi)$ is projective, therefore $\rho_\psi$ is projective as well. We deduce that $\rho_\psi$ is a progenerator of the category and since $\textup{End}_{R[H]}(\rho_\psi) = R$ by b), the category is semi-simple by Morita equivalence.

\noindent d) Set $H_1$, $H_2$ and $H$ for the groups appearing. We have a surjective group morphism
\begin{eqnarray*} 
H_1 \times H_2 & \rightarrow & H \\
((w_1,t_1),(w_2,t_2)) & \mapsto & (w_1 + w_2 , t_1 + t_2)\end{eqnarray*}
whose kernel is $\{((0,t),(0,-t)) \ | \ t \in F\}$. The representation $\rho_\psi^1 \otimes_R \rho_\psi^2$ factors through this group morphism, so it defines a representation of $H$. By taking compatible complete polarisations $W_1 = X_1 \oplus Y_1$ and $W_2 = X_2 \oplus Y_2$, we see that
$$\textup{ind}_{X_1 \times F}^{H_1}(\psi_{X_1}) \otimes_R \textup{ind}_{X_2 \times F}^{H_2}(\psi_{X_2}) \simeq \textup{ind}_{(X_1 \times F) \times (X_2 \times F)}^{H_1 \times H_2} (\psi_{X_1} \otimes_R \psi_{X_2}) \simeq \textup{ind}_{X \times F}^H(\psi_X).$$
Therefore $\rho_\psi^1 \otimes_R \rho_\psi^2 \in \textup{Rep}_R(H)$ is the Heisenberg representation associated to $\psi$. \end{proof}

\subsection{} We describe the change of models in the Heisenberg representation associated to $\psi$. Let $A_1$ and $A_2$ be self-dual subgroups in $W$. As $\mathcal{O}_F$ is local, principal and complete, the subgroup $A_1 + A_2$ has finite index in a closed subgroup of $W$ by \cite[I.C.5]{vig}. Therefore $A_1 + A_2$ is closed itself, so $(A_1 \cap A_2)^\perp = A_1 + A_2$, which improves \cite[Chap 2, Rem I.7]{mvw}. Let $\psi_{A_1}$ and $\psi_{A_2}$ be characters restricting to $\psi$, then we have the following explicit formula to construct an intertwining operator between $S_{A_1}$ and $S_{A_2}$, also called a change of models:

\begin{prop}  \label{prop:intertwining_change_of_models} Let $\mu$ be a Haar measure of $A_{1,H} \cap A_{2,H} \backslash A_{2,H}$ with values in $R$. Let $\omega \in W$ satisfying $\psi_{A_1}((a,0)) \psi_{A_2}((a,0))^{-1} = \psi(\langle a , \omega \rangle )$ for all $a \in A_1 \cap A_2$. Then the map $I_{A_1,A_2,\mu,\omega}$ associating to $f \in S_{A_1}$ the function
$$I_{A_1,A_2,\mu,\omega}f : h \longmapsto \int_{A_{1,H} \cap A_{2,H} \backslash A_{2,H}} \psi_{A_2}(a)^{-1} f((\omega,0) a h) \ d\mu(a)$$
is an isomorphism of representations in $\textup{Hom}_{R[H]}(S_{A_1},S_{A_2}) \simeq R$. \end{prop}

\begin{proof} Same proof as \cite[Chap 2, Lem I.7]{mvw}, plus Proposition \ref{prop:heisenberg_representation} b). \end{proof}

The expression of the intertwining operator becomes simpler when $\psi_{A_1}(a,t) = \psi(t)$ and $\psi_{A_2}(a,t) = t$, in which case $\omega = 0$ works. We can always choose the characters $\psi_{A_1}$ and $\psi_{A_2}$ of this form provided $p \neq 2$ or $A_1$ and $A_2$ are both Lagrangians. In this case, we simply obtain $I_{A_1,A_2,\mu} \in \textup{Hom}_{R[H]}(S_{A_1},S_{A_2})$ where 
$$I_{A_1,A_2,\mu} f : h \longmapsto \int_{A_1\cap A_2  \backslash A_2} f((a,0)h) d \mu(a).$$ 

\section{The $R$-metaplectic group and models of the Weil representation}

\subsection{} Let $\psi : F \to R^\times$ be a non-trivial character and let $(\rho_\psi,S) \in \textup{Rep}_R(H)$ be a model of the Heisenberg representation. The group $\textup{Sp}(W)$ acts on the Heisenberg group $H$ via
$$\begin{array}{ccc}
G \times H & \rightarrow & H \\
(g,(w,t)) & \mapsto & g \cdot (w,t) = (gw,t) \end{array}.$$
This action fixes the centre of $H$ and therefore preserves the isomorphism class of the Heisenberg representation. In other words, for $g \in \textup{Sp}(W)$, the representation $(\rho_\psi^g,S)$ defined by $\rho_\psi^g(h) = \rho_\psi(g^{-1} \cdot h)$ is another model of the Heisenberg representation associated to $\psi$. Theorem \ref{thm:modular_stone_von_neumann} ensures $\rho_\psi$ and $\rho_\psi^g$ are isomorphic. Therefore there exists $M_g \in \textup{GL}_R(S)$, which is unique up to a scalar thanks to Proposition \ref{prop:heisenberg_representation} b), such that $M_g \in \textup{Hom}_{R[H]}(\rho_\psi,\rho_\psi^g)$ \textit{i.e.} for all $h \in H$ we have
$$M_g \rho_\psi(h) M_g^{-1} = \rho_\psi^g(h).$$
Assume we have fixed $M_g$ as above for each $g \in \textup{Sp}(W)$. We then obtain a projective representation $\sigma_S$ of $\textup{Sp}(W)$ which does not depend on the choice of the $M_g$'s via
$$g \in \textup{Sp}(W) \mapsto \textsc{red}(M_g) \in \textup{PGL}_R(S).$$
We can lift $\sigma_S$ to an actual representation of a central extension of $\textup{Sp}(W)$ via the fibre product construction. Indeed, we consider
$$\xymatrix{
\widetilde{\textup{Sp}}_{\psi,S}^R(W) \ar[d]^{p_S} \ar[r]^{\omega_{\psi,S}} & \textup{GL}_R(S) \ar[d]^{\textsc{red}} \\
\textup{Sp}(W) \ar[r]^{\sigma_S} & \textup{PGL}_R(S) 
}$$
where $\widetilde{\textup{Sp}}_{\psi,S}^R(W) = \textup{Sp}(W) \times_{\textup{PGL}_R(S)} \textup{GL}_R(S)$ is the fibre product of the group morphisms $\sigma_S$ and $\textsc{red}$ above $\textup{PGL}_R(S)$. The group morphisms $p_S$ and $\omega_{\psi,S}$ above are respectively the first and second projections.

\begin{defi} We call $(\omega_{\psi,S},S)$ the Weil representation associated to $\psi$ and $S$. \end{defi}

The following proposition is rather straightforward, using Proposition \ref{prop:intertwining_change_of_models} and the fact that isomorphisms of central extensions are parametrised by characters and the group $\textup{Sp}(W)$ is perfect, unless $\textup{Sp}(W) \simeq \textup{SL}_2(\mathbb{F}_3)$.

\begin{prop} \label{prop:metaplectic_group_def} The group $\widetilde{\textup{Sp}}_{\psi,S}^R(W)$ is a central extension of $\textup{Sp}(W)$ by $R^\times$ \textit{i.e.}
$$ 1 \to R^\times \overset{i_S}{\to} \widetilde{\textup{Sp}}_{\psi,S}^R(W) \overset{p_S}{\to} \textup{Sp}(W) \to 1$$
is an exact sequence where $i_S : \lambda \mapsto (\textup{id}_W,\lambda \textup{id}_S)$ has central image.

If $S$ and $S'$ are two models of $\rho_\psi$, and $\phi : S \to S'$ is an isomorphism of representations, the isomorphism of central extensions
$$\begin{array}{cccl}	
  \widetilde{\textup{Sp}}_{\psi,S}^R(W) & \to & \widetilde{\textup{Sp}}_{\psi,S'}^R(W) \\
  (g,M) & \mapsto & (g, \phi M \phi^{-1}) 
\end{array}$$
does not depend on the choice of $\phi$. We denote it by $\Phi_{S,S'}$.

Except in the case $F=\mathbb{F}_3$ and $\textup{dim}_F W = 2$, there a unique isomorphism of central extensions between $\widetilde{\textup{Sp}}_{\psi,S}^R(W)$ and $\widetilde{\textup{Sp}}_{\psi,S'}^R(W)$, which is given by $\Phi_{S,S'}$.\end{prop}

In other words, the isomorphism class of $\widetilde{\textup{Sp}}^R_{\psi,S}(W)$ as a central extension is independent of $S$ and all such central extensions are canonically identified thanks to $\Phi_{S,S'}$.

\begin{defi} \label{def:metaplectic_group_isomorphism_class} We call this isomorphism class of central extensions the $R$-metaplectic group. By extension, any group in this isomorphism class is an $R$-metaplectic group. \end{defi}

\subsection{} Let $(\rho_d,\textup{Ind}_F^H(\psi))$ be the representation where $H$ acts on the right-hand side of functions. For all self-dual subgroup $A$ in $W$, we can embed the model of the Heisenberg representation $(\rho_\psi,S_A) = (\rho_\psi,\textup{ind}_{A_H}^H(\psi_A))$  as a subrepresentation of $\rho_d$.

The action of $g \in \textup{Sp}(W)$ on $H$ gives an isomorphism
$$\begin{array}{ccccl}	
I_g : & \textup{ind}_{A_H}^H(\psi_A) & \to & \textup{ind}_{g A_H}^H(\psi_A^g) \\
 & f & \mapsto & g \cdot f 
\end{array}$$
where $g \cdot f : h \mapsto f(g^{-1} \cdot h)$ and $\psi_A^g : a \in g A_H \mapsto \psi_A(g^{-1} \cdot h) \in R^\times$. Then for all $h \in H$
$$I_g \circ \rho_d(h) = \rho_d (g^{-1} \cdot h) \circ I_g.$$
Composing with the change of models $I_{gA,A,\mu,\omega}$ of Proposition \ref{prop:intertwining_change_of_models}, we obtain
$$S_A \overset{I_g}{\longrightarrow} S_{gA} \overset{I_{gA,A,\mu,\omega}}{\longrightarrow} S_A,$$
which satisfies $I_{gA,A,\mu,\omega} \circ I_g \circ \rho_d(h) = \rho_d(g^{-1} \cdot h) \circ I_{gA,A,\mu,\omega} \circ I_g$ for all $h \in H$. Hence
$$(g,I_{gA,A,\mu,\omega} \circ I_g) \in \widetilde{\textup{Sp}}_{\psi,S_A}^R(W).$$
In particular $I_{gA,A,\mu,\omega}$ is a multiple of the identity for all $g \in \textup{Stab}(A) \cap \textup{Stab}(\psi_A)$ \textit{i.e.}
$$g \in \textup{Stab}(A) \cap \textup{Stab}(\psi_A) \mapsto (g , I_g) \in \widetilde{\textup{Sp}}_{\psi,S_A}^R(W)$$
is a group morphism.

\subsection{} Here are the most commonly used explicit models of the Weil representation.

\subsubsection{Schr\"odinger model} Let $X$ be a Lagrangian. We consider the Schr\"odinger model $S_{\psi,X}$ associated to $\psi$ and $X$. We recall that the character $\psi_X$ is trivial on $X$.

Let $P(X)$ be the parabolic in $\textup{Sp}(W)$ stabilising $X$. Then as we have just remarked, the following map is a group morphism 
$$p \in P(X) \mapsto (p,I_p) \in   \widetilde{\textup{Sp}}_{\psi,S_X}^R(W).$$
Choosing a complete polarisation $W=X \oplus Y$, we identify $S_{\psi,X}$ with $C_c^\infty(Y)$. Then
$$I_p f : (y,0) \mapsto \psi(\frac{1}{2} \langle a^* y,b^* y\rangle ) f((a^* y,0)), \textup{ for } p = \left[ \begin{array}{cc} a & b \\
0 & (a^*)^{-1} \end{array} \right] \in P(X).$$
It provides an embedding of $M(X)$ and $N(X)$ in the $R$-metaplectic group.

We also have
$$(g,I_{Y,X,\mu_X} \circ I_g) \in \widetilde{\textup{Sp}}_{\psi,S_X}^R(W), \textup{ for } g = \left[ \begin{array}{cc} 0 & c \\
(c^*)^{-1} & 0 \end{array} \right] \in \textup{Sp}(W),$$
where $I_{Y,X,\mu_X}$ is simply a Fourier transform, so the composition is given by
$$I_{Y,X,\mu_X} \circ I_g f : (y,0) \mapsto \int_X \psi(\langle x,y\rangle ) f(c^{-1} x) d \mu_X (x).$$

\subsubsection{Mixed Schr\"odinger model} Let $0 \subsetneq X \subsetneq W$ be totally isotropic. Let $Y$ be a totally isotropic subspace in duality with $X$, we have a decomposition $W = X \oplus W^0 \oplus Y$ where $W^0$ is the orthogonal of the symplectic subspace $X \oplus Y$. Let $(\rho_\psi^0,S^0)$ be the Heisenberg representation of $H(W^0)$ associated to $\psi$. We realise the Heisenberg representation of $H(X \oplus Y)$ on the Schr\"odinger $C_c^\infty(Y)$ associated to $\psi$ and $X$. Proposition \ref{prop:heisenberg_representation} ensures that $S = C_c^\infty(Y) \otimes_R S^0$ is a model of the Heisenberg representation of $H(W)$ associated to $\psi$. Let $P(X)$ be the stabiliser of $X$ in $\textup{Sp}(W)$ and $j : P(X) \to \textup{Sp}(W^0)$ the projection to the symplectic part of the Levi $M(X)$ of $P(X)$. We have a natural section of $j$ via the inclusion of $\textup{Sp}(W^0)$ in $M(X)$. This embedding induces an isomorphism of groups $p \in P(X) \mapsto (p u^{-1},u) \in \textup{Ker}(j) \rtimes \textup{Sp}(W^0)$.

\begin{lem} We have an isomorphism of groups
$$\begin{array}{cccl}	
P(X) \times_{\textup{Sp}(W^0)}  \widetilde{\textup{Sp}}_{\psi,S^0}^R(W^0)  & \overset{\sim}{\to} & p_S^{-1}(P(X)) \\
(p,\tilde{u}) & \mapsto & (p , I_{p u^{-1}} \otimes \omega_{\psi,S^0} (\tilde{u}))
\end{array}$$
where the fibre product on the left-hand side is given by $j$ and $p_{S^0}$. \end{lem}

In particular we can consider the action of $\textup{Ker}(j)$ via the group morphism
$$p \in \textup{Ker} (j) \mapsto (p,I_p \otimes \textup{Id}_{S^0}) \in \widetilde{\textup{Sp}}_{\psi,S}^R(W).$$
We give the actions of subgroups/subsets of interest. Let $f \in C_c^\infty(Y) \otimes S^0 = C_c^\infty(Y,S^0)$.
\begin{itemize}[label=$\bullet$]
\item For all $p=(a,u) \in M(X) = \textup{GL}(X) \times \textup{Sp}(W^0)$, we have
$$(p,\tilde{u}) \cdot f : y \mapsto \omega_{\psi,S^0}(\tilde{u}) \cdot (f (a^* y)).$$
\item For all $p = \left[ \begin{array}{ccc} \textup{Id}_X & 0 & s \\
0 & \textup{Id}_{W^0} & 0 \\
0 & 0 & \textup{Id}_Y \end{array} \right] \in \textup{Sp}(W)$, we have
$$(p,(\textup{Id}_{W^0},\textup{Id}_{S^0})) \cdot f : y \mapsto \psi( \frac{1}{2} \langle s y , y \rangle ) f(y).$$
\item For all $p = \left[ \begin{array}{ccc} \textup{Id}_X & v & 0 \\
0 & \textup{Id}_{W^0} & - v^* \\
0 & 0 & \textup{Id}_Y \end{array} \right] \in \textup{Sp}(W)$, we have
$$(p,(\textup{Id}_{W^0},\textup{Id}_{S^0})) \cdot f : y \mapsto  \rho_\psi^0((v^* y,0)) \cdot (f(y)).$$
\end{itemize}

\subsubsection{Lattice model} Let $F$ be local non archimedean and assume $p \neq 2$. Let $A$ be a self-dual lattice in $W$. Since $p \neq 2$, we can set $\psi_A(a,t) = \psi(t)$ and therefore choose $\omega = 0$ in Proposition \ref{prop:intertwining_change_of_models}. For all $g \in \textup{Sp}(W)$, the set $A/gA \cap A$ is finite and we endow it with the counting measure $\mu$. An explicit computation gives
$$I_{gA,A,\mu} \circ I_g f : (w,0) \mapsto \sum_{a \in A/gA \cap A} \psi(\frac{1}{2} \langle a,w \rangle) f((g^{-1}(a+w),0)).$$
If $K$ is the stabiliser of $A$ in $\textup{Sp}(W)$, we have a group morphism
$$k \in K \mapsto (k,I_k) \in  \widetilde{\textup{Sp}}_{\psi,S_A}(W)$$
which is a smooth representation $k \in K \mapsto \omega_{\psi,S_A}((k,I_k))=I_k \in \textup{GL}_R(S_A)$.

\begin{rem} The lattice model also exists when $p=2$, except that the character $\psi_A$ can't be extended trivially to $A$ and the formulas usually become inoperable. \end{rem}

\subsubsection{Another model} Let $(\rho_\psi,S)$ be a model of the Heisenberg representation of $H$ and let $g \in \textup{Sp}(W)$. For all $s \in S$, the function
$$w \in W \mapsto \psi(\frac{\langle w,g^{-1} w \rangle }{2}) \rho_\psi((\textup{Id}_W-g^{-1})w,0) s \in S$$
is invariant under $\textup{Ker}(\textup{Id}_W-g^{-1})$ \textit{i.e.} factors through a function on $W/\textup{Ker}(\textup{id}_W - g^{-1})$.

\begin{lem} Let $g \in \textup{Sp}(W)$ and let $\mu_g$ be a Haar measure of $W / \textup{Ker}(\textup{Id}_W-g^{-1})$.

\begin{itemize}[label=$\bullet$]
\item If $F$ is finite, we define $M[g] \in \textup{End}_R(S)$ by
$$M[g] : s \mapsto \int_{W/\textup{Ker}(1-g^{-1})}\psi(\frac{\langle w,g^{-1} w \rangle }{2}) \rho_\psi  ((\textup{Id}_W-g^{-1})w,0)) s \ d\mu_g (w).$$
\item If $F$ is local non-archimedean, for all lattice $L$ in $W/\textup{Ker}(\textup{Id}_W-g^{-1})$, define
$$M_L[g] : s \mapsto \int_L \psi(\frac{\langle w,g^{-1} w \rangle }{2}) \rho_\psi  ((\textup{Id}_W-g^{-1})w,0)) s \ d \mu_g w$$
For all $s \in S$, there exist a lattice $L_s$ and an element $M[g] s \in S$ such that 
$$M_L[g] s=M[g] s \textup{ for all lattices } L \supseteq L_s.$$
In this sense $M_L[g]s$ is independent of $L$ and $M[g] : s \mapsto M[g] s$ is in $\textup{End}_R(S)$.
\end{itemize}
Then $M[g] \in \textup{Hom}_H(\rho_\psi,\rho_\psi^g)$ \textit{i.e.} $(g,M[g]) \in \widetilde{\textup{Sp}}_{\psi,S}^R(W)$.
\end{lem}

\begin{proof} Same proof as \cite[Chap 2, Lem II.2]{mvw}. Note that, for $L$ a lattice in $W$, it is elementary to check that the following map belongs to $\textup{Hom}_{L \times F}(\rho_\psi,\rho_\psi^g)$
$$s \mapsto \int_L \rho_{\psi}^g(w,0)^{-1} \rho_{\psi}(w,0) s d \mu_g w.$$
Furthermore $\displaystyle \rho_\psi^g((w,0))^{-1} \rho_\psi((w,0)) = \psi(\frac{\langle w , g^{-1} w \rangle}{2})\rho_\psi(((\textup{Id}_W-g^{-1}) w,0))$. \end{proof}

\begin{lem} \label{lem:centraliser_metaplectic} Let $g_1, g_2 \in \textup{Sp}(W)$. Assume that $g_1 g_2 = g_2 g_1$. Then
$$M[g_1] M[g_2] = M[g_2] M[g_1].$$ \end{lem}

\begin{proof} Same proof as \cite[Chap. 2, Lem. II.5]{mvw}. \end{proof}

\section{Properties of the $R$-metaplectic group and the Weil representation} \label{sec:props_of_met_group_and_weil_rep}

Let $S$ be a model of the Heisenberg representation associated to $\psi : F \to R^\times$.

\subsection{} If $G$ is a group, we denote by $[G,G]$ its derived subgroup. We have the following properties for the $R$-metaplectic group:

\begin{theo} \label{thm:metaplectic_group} Let $\widehat{\textup{Sp}}_{\psi,S}^R(W)$ be the derived subgroup of $\widetilde{\textup{Sp}}_{\psi,S}^R(W)$.
\begin{enumerate}[label=\textup{\alph*)}]
\item If $F$ is finite, or if the characteristic $\ell$ of $R$ is $2$, there exists a section of $p_S$
$$\textup{Sp}(W) \to \widetilde{\textup{Sp}}_{\psi,S}^R(W).$$
Except in the exceptional case $F= \mathbb{F}_3$ and $\textup{dim}_F W= 2$, this group morphism is unique. This embedding of $\textup{Sp}(W)$ induces a group isomorphism
$$\widehat{\textup{Sp}}_{\psi,S}^R(W) \simeq [\textup{Sp}(W),\textup{Sp}(W)].$$
Here $[\textup{Sp}(W),\textup{Sp}(W)] = \textup{Sp}(W)$, except in the exceptional case.
\item If $F$ is local non-archimedean and $\ell \neq 2$, such a section of $p_S$ does not exist. However, the derived subgroup fits into the exact sequence
$$1 \to \{ \pm 1 \} \overset{i_S}{\to} \widehat{\textup{Sp}}_{\psi,S}^R(W) \overset{p_S}{\to} \textup{Sp}(W) \to 1.$$
The group $\widehat{\textup{Sp}}_{\psi,S}^R(W)$ is the unique two-fold cover of $\textup{Sp}(W)$ contained in $\widetilde{\textup{Sp}}_{\psi,S}^R(W)$.
\item The group $\widehat{\textup{Sp}}_{\psi,S}^R(W)$ is perfect, except in the exceptional case. \end{enumerate} \end{theo}

\begin{proof} a) If such a group morphism $\sigma$ exists, it induces an isomorphism of central extensions
$$(g,\lambda) \in \textup{Sp}(W) \times R^\times \to \sigma(g) i_S(\lambda) \in \widetilde{\textup{Sp}}_{\psi,S}^R(W).$$
Since two such isomorphisms differ by a character, we deduce that $\sigma$ is unique, except in the exceptional case. Moreover, the derived subgroup of the left-hand side is $[\textup{Sp}(W),\textup{Sp}(W)]$. There remains to prove the existence of $\sigma$. In the finite case, it is a consequence of \cite[Th 3.3]{stei} as the symplectic group is its own universal covering in the sense of \cite{moore}. This means that any central extension of the symplectic group splits. We deal with the non-archimedean local field case when $\ell = 2$ in the next paragraph.

b) Suppose $F$ is local non-archimedean. When $R=\mathbb{C}$, there exists by \cite{weil} a character $A_\mathbb{C} \to \mathbb{C}^\times$ of a central extension $A_\mathbb{C}$ of $\textup{Sp}(W)$ by $\mathbb{C}^\times$ whose kernel is $\widehat{\textup{Sp}}(W)$, the unique non-trivial central extension of $\textup{Sp}(W)$ by $\{\pm 1\}$. The latter group is perfect. It is the derived subgroup of $A_\mathbb{C}$. This result was generalised in \cite[\S 5]{ct} when $R$ is an integral domain of characteristic $\ell \neq p$ and such that there exists a non-trivial character $\psi : F \times R^\times$. In particular our coefficient field satisfies these assumptions. We deduce there exists a character $A_R \to R^\times$ of a central extension $A_R$ of $\textup{Sp}(W)$ by $R^\times$ whose kernel is $\widehat{\textup{Sp}}(W)$ if $\ell \neq 2$ and $\textup{Sp}(W)$ if $\ell = 2$. These two groups are perfect. We refer to \cite[Ann A.2]{trias_thesis} for details about the identification well-known of the experts between $A_R$ and the $R$-metaplectic group.

We obtain a character $\widetilde{\textup{Sp}}_{\psi,S}^R(W) \to R^\times$ whose kernel is a perfect group, and contains the derived subgroup, therefore it is equal to the derived subgroup. When $\ell = 2$, this group is $\textup{Sp}(W)$ and the same argument as in the finite case shows that the section of $p_S$ thus obtained is unique. When $\ell \neq 2$, this is the non-trivial two-fold cover of $\textup{Sp}(W)$. Moreover the $R$-metaplectic group does not contain $\textup{Sp}(W)$ as a subgroup according to \cite[Th 5.4]{ct}, so any two-fold cover of $\textup{Sp}(W)$ in the $R$-metaplectic group is unique. 

c) The symplectic group and its non-trivial double cover are known to be perfect. \end{proof}

We endow $R$ with the discrete topology. Let $S$ be a vector space over $R$. We endow it with the discrete topology. The compact-open topology on $\textup{GL}_R(S)$ is generated by the prebasis $S_{s,s'} = \{ g \in \textup{GL}_R(S) \ | \ g s = s' \}$ for $s$ and $s'$ running over $S$. Then, a representation $S$ of a topological group $G$ is smooth if and only if the associated group morphism $G \to \textup{GL}_R(S)$ is continuous.

\begin{prop} \label{prop:metaplectic_group_topological_fibre_product} The $R$-metaplectic group $\widetilde{\textup{Sp}}_{\psi,S}^R(W)$  is the fibre product in the category of topological groups of the continuous morphisms $\sigma_S$ and $\textsc{red}$. It is a topological subgroup of $\textup{Sp}(W) \times \textup{GL}_R(S)$ and a topological central extension of $\textup{Sp}(W)$ by $R^\times$. Moreover the isomorphisms $\Phi_{S,S'}$ above are isomorphisms of topological central extensions. \end{prop}

\begin{proof} When $F$ is finite, the topology is discrete, so the groups appearing are all topological groups and all maps are continuous. Let $F$ be local non-archimedean. First of all, the map $\textsc{red}$ is continuous by definition of the quotient topology.

There remains to prove that $\sigma_S$ is continuous. Note that $\Phi_{S,S'}$ induces an isomorphism of topological groups $M \in \textup{GL}_R(S) \to \phi M \phi^{-1} \in \textup{GL}_R(S')$, therefore it is enough to find one $S$ such that $\sigma_S$ is continuous since $\phi \sigma_S(g) \phi^{-1} =\sigma_{S'}(g)$ for $g \in \textup{Sp}(W)$.

\begin{lem} \label{lem:continuity_of_sigma_S} Let $L$ be a self-dual lattice in $W$ and let $S_L$ be the lattice model associated to $L$ and $\psi$. Then $\sigma_{S_L}$ is continuous. \end{lem}

\begin{proof} Let $K$ be the stabiliser of $L$ in $\textup{Sp}(W)$. It is compact open subgroup. Let $k \in K$ and consider the linear map
$$\begin{array}{ccccl}	
N_k : & \textup{ind}_{L_H}^H(\psi_L) & \to & \textup{ind}_{L_H}^H(\psi_L^k) \\
& f & \mapsto & k \cdot f 
\end{array}$$
where $k \cdot f : h \mapsto f(k^{-1} \cdot h)$ and $\psi_L^k : (l,t) \mapsto \psi_L((k^{-1} l, t))$. 

When $p \neq 2$, we can choose $\psi_L$ such that $\psi_L(l,t) = \psi(t)$.  Thus $\psi_L^k = \psi_L$ for all $k \in K$ and $N_k \in \textup{GL}_R(S_L)$ satisfies
$$N_k \circ \rho_\psi = \rho_\psi^{k^{-1}} \circ N_k,$$
or equivalently $N_k \circ \rho_\psi^k = \rho_\psi \circ N_k$. Therefore $(N,S_L)$ where $N : k \mapsto N_k$ is a representation of $K$ lifting $\sigma_{S_L}$ in the sense that $\sigma_S(k) = \textsc{red} (N_k)$ for all $k \in K$. Moreover $N$ is a smooth representation because the action of $k \in K$ on $f \in S_L \simeq C_c^\infty(W)$ is given by $k \cdot f (w) = f(k^{-1} w)$ and is easily checked to be smooth.

When $p = 2$, the character $\psi$ can't be extended trivially to $L \times \{0\}$. However, there exists $n \in \mathbb{N}$ such that $\varpi_F^n L \times \textup{Ker}(\psi)$ is a subgroup of $\textup{Ker}(\psi)$. Fix such a $n$ and let $K_n$ be the kernel of the reduction morphism $K \to \textup{GL}(L/\varpi_F^n L)$. Then the same arguments apply to $K_n$, namely $N : k \in K_N \mapsto N_k \in \textup{GL}_R(S_L)$ which is smooth and lifts $\sigma_{S_L}$.

We have shown there exists a smooth representation $N$ lifting $\sigma_{S_L}$ of a compact open subgroup $K$ of $\textup{Sp}(W)$ \textit{i.e.} $\sigma_{S_L} = \textsc{red} \circ N$. Since $N$ and $\textsc{red}$ are continuous, the group morphism $\sigma_{S_L}$ is continuous on $K$. This implies $\sigma_{S_L}$ is continuous on $\textup{Sp}(W)$. \end{proof}

Therefore $\widetilde{\textup{Sp}}_{\psi,S}^R(W)$ is the fibre product in the category of topological groups of the continuous group morphisms $\sigma_S$ and $\textsc{red}$ above $\textup{PGL}_R(S)$. By definition, this fibre product is a topological subgroup of the product $\textup{Sp}(W) \times \textup{GL}_R(S)$. As we already remarked, the isomorphism of central extensions $\Phi_{S,S'}$ are homeomorphisms. \end{proof}

Because the second projection $\omega_{\psi,S} : \widetilde{\textup{Sp}}_{\psi,S}^R(W) \to \textup{GL}_R(S)$ is continuous, we obtain:

\begin{cor} \label{cor:weil-rep-is-smooth} The Weil representation $(\omega_{\psi,S},S)$ is smooth. \end{cor}

\subsection{} The first projection $p_S$ is also continuous and defines a fibre bundle:

\begin{prop} The map $p_S : \widetilde{\textup{Sp}}_{\psi,S}^R(W) \to \textup{Sp}(W)$ has local trivialisations and this turns the $R$-metaplectic group into a trivial fibre bundle of basis $\textup{Sp}(W)$ and fibre $R^\times$. \end{prop}

\begin{proof} Since the base $\textup{Sp}(W)$ is locally profinite, any fibre bundle over $\textup{Sp}(W)$ is trivial. So we simply show that $p_S$ admits local trivialisations since the fibres of $p_S$ are all $R^\times$. In the proof of Lemma \ref{lem:continuity_of_sigma_S}, we found a continuous group morphism $N : K \to \textup{GL}_R(S)$ with $K$ a compact open subgroup of $\textup{Sp}(W)$. As the embedding $K \hookrightarrow \textup{Sp}(W)$ is clearly continuous, the universal property of the fibre product provides a continuous embedding of $K$ in the $R$-metaplectic group inducing a local trivialisation $K \times R^\times \simeq p_S^{-1}(K)$. As $p_S$ is a continuous group morphism, it admits local trivialisations everywhere. \end{proof}

We deduce that:

\begin{cor} The $R$-metaplectic group $\widetilde{\textup{Sp}}_{\psi,S}^R(W)$ is a locally profinite group and its derived subgroup $\widehat{\textup{Sp}}_{\psi,S}^R(W)$ is open. \end{cor}

\begin{proof} We use the trivialisation $K \times R^\times \simeq p_S^{-1}(K)$ from the previous proof to obtain an embedding of $K$ in the $R$-metaplectic group as an open subgroup. Since this image of $K$ is open, and compact, there exists a basis of neighbourhood of the identity made of open compact subgroups. Note that the quotient of the $R$-metaplectic group by its derived subgroup is the discrete group $R^\times / \{ \pm 1 \}$. Therefore the derived subgroup is open. \end{proof}

\subsection{} Let $X$ be Lagrangian in $W$. Let $S_X$ be the model of the Heisenberg representation associated to $\psi$ and $X$. The formulas of the Schr\"odinger model give:

\begin{prop} We have $p_{S_X}^{-1}(P(X)) \simeq P(X) \times R^\times$. \end{prop}

\noindent In particular any subgroup of $P(X)$ is split in the $R$-metaplectic group. Furthermore, similarly to \cite[Chap 2, Lem II.9]{mvw}, there exists a unique splitting of $N(X)$ in the $R$-metaplectic group valued in the derived subgroup and normalised by $P(X)$. Moreover:

\begin{prop} Let $F$ be local non-archimedean. Let $A$ be a self-dual lattice in $W$. The natural embedding $k \in K \mapsto (k,I_k) \in \widetilde{\textup{Sp}}_{\psi,S_A}^R(W)$ of the open compact subgroup $K = \textup{Stab}(A) \cap \textup{Stab}(\psi_A)$ of $\textup{Sp}(W)$ has image in $\widehat{\textup{Sp}}_{\psi,S_A}^R(W)$. \end{prop}

\begin{proof} Same proof as \cite[Chap. 2, Lem. II.10]{mvw}. \end{proof}

\subsection{} Let $\widetilde{\textup{Sp}}_\psi^R(W)$ be the $R$-metaplectic group associated to $\psi$ from Definition \ref{def:metaplectic_group_isomorphism_class}. Let $S$ be a model of the Heisenberg representation associated to $\psi$. By definition, there exists an isomorphism of central extensions $\varphi_S$ between $\textup{Mp}_\psi^R(W)$ and $\widetilde{\textup{Sp}}_{\psi,S'}^R(W)$. In addition $\varphi_S$ is unique, unless $\textup{Sp}(W) \simeq \textup{SL}_2(\mathbb{F}_3)$. For any other model $S'$, we fix isomorphisms between $\widetilde{\textup{Sp}}_\psi^R(W)$ and $\widetilde{\textup{Sp}}_{\psi,S'}^R(W)$ by setting $\varphi_{S'} = \Phi_{S,S'} \circ \varphi_S$.

\begin{defi} The representations $(\omega_{\psi,S} \circ \varphi_S,S)$ and $(\omega_{\psi,S'} \circ \varphi_{S'},S')$ of $\textup{Rep}_R(\widetilde{\textup{Sp}}_\psi^R(W))$ are isomorphic and we call their isomorphism class the Weil representation associated to $\psi$. By extension, any model in this isomorphism class is also the Weil representation. \end{defi}

\begin{rem} In the exceptional case, the identification $\varphi_S$ is not unique. Therefore the Weil representation depends on the identification given by $\varphi_S$. \end{rem}

\subsection{} Let $\omega_{\psi,W}^R$, or simply $\omega_\psi^R$, be the Weil representation associated to $\psi$.

\begin{prop} The smooth representation $\omega_\psi^R$ is and admissible. \end{prop}

\begin{proof} We first recall that $\omega_\psi^R$ is smooth by Corollary \ref{cor:weil-rep-is-smooth}. We now prove admissibility. Let $F$ be local non-archimedean and let $A$ be a self-dual lattice in $W$. By definition $S_A = \textup{ind}_{A \times F}^{W \times F}(\psi_A)$ where $\psi_A$ extends $\psi$ to $A \times F$. Let $K$ be a compact open subgroup in $\textup{Stab}(A)$. For all $f \in S_A^K$, let $L$ be a lattice in $W$ such that $f$ is $L$-bi-invariant, \textit{i.e.} $f(l+w,0) = f(w,0)$ for all $l \in L$ and all $w \in W$, and for all $k \in K$ and all $l \in L$, we have $\psi_A(k^{-1} l , 0 ) = 1$. We assume $L \subseteq A$, up to replacing $L$ by $L \cap A$. Then, for all $l \in L$ and all $w \in W$ and all $k \in  K$, we have
\begin{eqnarray*} f((w,0))=f((w+l,0))=f(k^{-1}(l+w,0)) &=& f((k^{-1} l,\frac{1}{2}\langle k^{-1} w, k^{-1} l \rangle ) (k^{-1}w,0)) \\
&=& \psi_A((k^{-1}l,0)) \psi(\frac{1}{2}\langle w,l \rangle ) f(k^{-1}(w,0)) \\
 &=& \psi(\frac{1}{2}\langle w,l \rangle ) f((w,0)) . \end{eqnarray*}
We deduce that $\textup{supp}(f)$ is included in $2 L^\perp$ where $L^\perp$ appears in Lemma \ref{lem:stonve_von_neumann_existence}. The vector space $S_A^K$ has dimension at most $|(A \times F) \backslash (2 L^\perp \times F)/K|$, which is finite. \end{proof}

\subsection{} Let $\widehat{\textup{Sp}}_\psi^R(W)$ be the derived subgroup of $\widetilde{\textup{Sp}}_\psi^R(W)$. Let $Z$ be the centre of $\widehat{\textup{Sp}}_\psi^R(W)$ and let $Z' = \{ z \in Z \ | \ z^2 = 1 \}$. The quotient morphism induces an isomorphism
$$\widetilde{\textup{Sp}}_\psi^R(W) / \widehat{\textup{Sp}}_\psi^R(W) \simeq Z / Z' =R^\times / \{ \pm 1 \}.$$
Note that the square map $\lambda \in R^\times \mapsto \lambda^2 \in R^\times$ factors through $R^\times /\{ \pm 1 \} \to R^\times$. We let $\chi^2 : \widetilde{\textup{Sp}}_\psi^R(W) \to R^\times /\{ \pm 1 \} \to R^\times$ be the composition with the quotient morphism. In particular $\chi^2(\hat{g} \lambda)=\lambda^2$ for $\hat{g} \in \widehat{\textup{Sp}}_\psi^R(W)$ and $\lambda \in R^\times$.

The next two propositions are consequences of Proposition \ref{prop:heisenberg_representation}, using a) and d).

\begin{prop} \label{prop:weil_rep_contragredient} We have $(\omega_\psi^R)^\vee \simeq \omega_{\psi^{-1}}^R \otimes \chi^2$ in $\textup{Rep}_R(\widetilde{\textup{Sp}}_\psi^R(W))$. By restricting to the derived subgroup, we obtain
$$(\omega_\psi^R)^\vee \simeq \omega_{\psi^{-1}}^R \textup{ in } \textup{Rep}_R(\widehat{\textup{Sp}}_\psi^R(W)).$$ \end{prop}

\begin{prop} If $W = W_1 \oplus W_2$ is an orthogonal sum, there exists a unique (resp. canonical, in the exceptional case) group morphism
$$i_{W_1,W_2} : \widetilde{\textup{Sp}}_{\psi,S_1}^R(W_1) \times \widetilde{\textup{Sp}}_{\psi,S_2}^R(W_2)  \to \widetilde{\textup{Sp}}_{\psi,S}^R(W)$$
which lifts the embedding $\textup{Sp}(W_1) \times \textup{Sp}(W_2) \to \textup{Sp}(W)$ and commute to the fibre product projections. Its kernel $\{ ((1_{\textup{Sp}(W_1)},\lambda \textup{Id}_{S_2}),(1_{\textup{Sp}(W_2)},\lambda^{-1} \textup{Id}_{S_2})) \ | \ \lambda \in R^\times \}$ is isomorphic to $R^\times$ embedded anti-diagonally. We obtain by pullback a representation
$$\omega_{\psi,W}^R \circ i_{W_1,W_2} \in \textup{Rep}_R(\widetilde{\textup{Sp}}_{\psi,S_1}^R(W_1) \times \widetilde{\textup{Sp}}_{\psi,S_2}^R(W_2))$$
which is isomorphic to $\omega_{\psi,W_1}^R \otimes \omega_{\psi,W_2}^R$.  \end{prop}

\section{Non-normalised Weil factor} \label{sec:non_norm_Weil_factor}

We define in this section a non-normalised version of the Weil factor. Our construction has the benefit to be more elementary and more direct than the construction of the usual Weil factor. We then relate our factor to the usual Weil factor of \cite{weil,perrin,rao} when $R=\mathbb{C}$, and to its generalisation in \cite{ct} to coefficient fields $R$ containing a square root of $q$.

\subsection{} Let $X$ be a vector space of finite dimension $m$ over $F$. Since the pro-order of $X$ is a power of $p$ and the characteristic $\ell$ of $R$ is different from $p$, there exists a Haar measure $\mu$ of $X$ with values in $R$ \cite[I.2.4]{vig}. We recall that a quadratic form $Q$ over $X$ is non-degenerate if its radical $\textup{rad} (Q) = \{ x \in X \ | \ Q(x+y)-Q(y)-Q(x) = 0, \textup{ for all } y \in X\}$ is reduced to $0$. Assume there exists a non-trivial smooth additive character $\psi : F \to R^\times$.

\begin{prop} Let $Q$ be a non-degenerate quadratic form over $X$. There exists a unique element $\Omega_\mu(\psi \circ Q) \in R^\times$ such that for all $f \in C_c^\infty(X)$ we have
$$\int_X \int_X f(y-x) \psi(Q(x)) d \mu (x) d \mu (y) = \Omega_\mu(\psi \circ Q) \int_X f(x) d \mu (x).$$ \end{prop}

\begin{proof} It is elementary to check that the linear form 
$$\mu' : f \in C_c^\infty(X) \mapsto \int_X \int_X f(y-x) \psi( Q(x)) d\mu (x) d\mu (y) \in R$$
is a (non-zero) Haar measure of $X$ with values in $R$. By uniqueness of the Haar measure \cite[I.2.4]{vig}, there exists a unique element $c \in R^\times$ such that $\mu' = c \mu$. \end{proof}

As the notation suggests, the factor $\Omega_\mu(\psi \circ Q)$ depends on $\mu$. We now extend the definition of our factor to degenerate quadratic forms. Note that any quadratic form $Q$ over $X$ induces a non-degenerate quadratic form $Q_{nd}$ over $X_Q = X / \textup{rad}(Q)$.

\begin{defi} Let $Q$ be a quadratic form over $X$. For $\mu$ a Haar measure of $X_Q$ with values in $R$, we call $\Omega_\mu(\psi \circ Q) = \Omega_\mu(\psi \circ Q_{nq})$ the non-normalised Weil factor. \end{defi}

\subsection{} Before reviewing some properties of this factor, we need to introduce a few notations. If $X'$ is a vector space isomorphic to $X$, we denote by $\textup{Iso}_F(X,X')$ the $F$-linear isomorphisms between $X$ and $X'$. We write $\textup{Aut}_F(X)$ for $\textup{Iso}_F(X,X)$. If $X^* = \textup{Hom}_F(X,F)$ is the dual of $X$, the set $\textup{Iso}_F(X,X^*)$ has a natural involution $\rho \mapsto \rho^*$ given by duality \textit{i.e.} $\rho^*(x) : y \in X \mapsto \rho(y)(x) \in R$ for $x \in X$. The symmetric morphisms	 are the fixed point for this operation and we denote them by
$$\textup{Iso}_F^{\textup{sym}}(X,X^*) = \{ \rho \in \textup{Iso}_F(X,X^*) \ | \ \rho=\rho^* \}.$$
If $\rho \in \textup{Iso}_F^{\textup{sym}}(X,X^*)$, we define a quadratic form $Q_\rho$ over $X$ by setting $Q_\rho(x) = \rho(x)(x)$ for $x \in X$. The sets of linear isomorphisms aboove inherit a natural locally profinite topology from the finite dimensional vector space $\textup{Hom}_F(X,X')$.

Let $\mu$ be a Haar measure of $X$ with values in $R$. For $\phi \in \textup{Aut}_F(X)$, the measure $\phi \cdot \mu = \mu \circ \phi^{-1}$ is a Haar measure of $X$. It is a scalar multiple of $\mu$ by uniqueness of the Haar measure. Denote by $|\phi| \in R^\times$ the unique scalar such that $\phi \cdot \mu = |\phi| \mu$. Then $|\phi|$ does not depend on the choice of $\mu$, we call it the modulus of $\phi$. For any compact open subgroup $K$ of $X$, we have
$$|\phi| = \frac{\textup{vol}_{\phi \cdot \mu}(K)}{\textup{vol}_\mu(K)} = \frac{\mu(\phi^{-1}(K))}{\mu(K)}.$$
Moreover the modulus $\phi \in \textup{Aut}_F(X) \mapsto |\phi| \in R^\times$ defines a smooth character. We also have $|\phi| = |\textup{det}_F(\phi)|_F$ for $\phi \in \textup{Aut}_F(X) = \textup{GL}_F(X)$. In particular $|\phi| \in q^\mathbb{Z}$.

We can define a modulus map on $\textup{Iso}_F(X,X^*)$ in the following way. Let $\mu$ be a Haar measure of $X$ with values in $R$. Let $\mu^*$ be its dual measure \textit{i.e.} the unique Haar measure of $X^*$ such that the Fourier transform
$$\begin{array}{ccccc} \mathcal{F}_{\mu} &:& C_c^\infty(X) & \to & C_c^\infty(X^*)  \\
 & & f & \mapsto & \mathcal{F}_{\mu} f
\end{array} \textup{ where } \mathcal{F}_{\mu} f : x^* \mapsto \int_X \psi(x^*(x)) f(x) d \mu(x)$$
has inverse
$$\begin{array}{ccccc} \mathcal{F}_{\mu^*} &:& C_c^\infty(X^*) & \to & C_c^\infty(X)  \\
 & & h & \mapsto & \mathcal{F}_{\mu^*} h
\end{array} \textup{ where } \mathcal{F}_{\mu^*} h : x \mapsto \int_{X^*} \psi(-x^*(x)) h(x^*) d \mu^*(x^*).$$
For $\rho \in \textup{Iso}_F(X,X^*)$, the measure $\rho \cdot \mu$ is a Haar measure of $X^*$, so there exists $|\rho|_\mu \in R^\times$ such that $\rho \cdot \mu = |\rho|_\mu \mu^*$ by uniqueness. As the notation suggests, the modules $|\rho|_\mu$ does depend on the choice of $\mu$, but only up to a square in $R^\times$. For any compact open subgroup $K$ of $X^*$, we have
$${|\rho|}_\mu = \frac{\textup{vol}_{\rho \cdot \mu}(K)}{\textup{vol}_{\mu^*}(K)} = \frac{\mu(\rho^{-1} (K))}{\mu^*(K)}.$$
Moreover $\rho \in \textup{Iso}_F(X,X^*) \mapsto |\rho|_{\mu} \in R^\times$ is locally constant. This modulus map is compatible with that of $\textup{Aut}_F(X)$ in the sense that $|\rho \circ \phi|_\mu = |\rho|_\mu \cdot |\phi|$ for $\phi \in \textup{Aut}_F(X)$ and $\rho \in \textup{Iso}_F(X,X^*)$. It is also invariant under duality \textit{i.e.} $|\rho|_\mu = |\rho^*|_\mu$. When $K$ is a lattice of $X$, \textit{i.e.} an open compact subgroup endowed with a structure of $\mathcal{O}_F$-module, we have $|\rho|_{\mu_K} \in q^\mathbb{Z}$.

\begin{prop} \label{prop:non_normal_weil_factor} Let $Q$ be a quadratic form over $X$. Let $\mu$ be a Haar measure of $X_Q$.
\begin{enumerate}[label=\textup{\alph*)}]
\item If $Q$ is the zero quadratic form, then $\Omega_\mu(\psi \circ 0) = \mu( \{ 0 \} )$.
\item For all $\lambda \in R^\times$, we have:
$$\Omega_{\lambda \mu} (\psi \circ Q) = \lambda \times \Omega_\mu(\psi \circ Q).$$
\item \label{non_invariance_par_isometrie_weil_non_norm_pt} If $X'$ is a vector space of the same dimension as $X$ and $\phi \in \textup{Iso}_F(X,X')$, we define the quadratic form $Q_\phi = Q \circ \phi^{-1}$ over $X'$. Its radical is $\phi(\textup{rad}(Q))$. Then
$$\Omega_{\phi \cdot \mu}( \psi \circ Q_\phi ) =  \Omega_\mu(\psi \circ Q).$$
In particular, if $\phi \in \textup{Aut}_F(X)$ preserves the radical of $Q$ \textit{i.e.} $\phi(\textup{rad}(Q))=\textup{rad}(Q)$, we have
$$\Omega_\mu( \psi \circ Q_\phi ) = |\phi|^{-1} \times \Omega_\mu(\psi \circ Q).$$
\item Let $Q_1 \oplus Q_2$ be the sum of two quadratic form $Q_1$ over $X_1$ and $Q_2$ over $X_2$, together with Haar measures $\mu_1$ of $(X_1)_{Q_1}$ and $\mu_2$ of $(X_2)_{Q_2}$ , then
$$\Omega_{\mu_1 \otimes \mu_2}(\psi \circ (Q_1 \oplus Q_2)) = \Omega_{\mu_1}(\psi \circ Q_1) \Omega_{\mu_2}(\psi \circ Q_2).$$
\item  The map $\rho \in \textup{Iso}_F^{\textup{sym}}(X,X^*) \mapsto \Omega_\mu(\psi \circ Q_\rho) \in R^\times$ is locally constant.
\item \label{facteur_de_weil_non_nomr_vs_classique_pt} Suppose $R$ contains a square root of $q$. Fix $q^{\frac{1}{2}} \in R^\times$. Let $\rho \in \textup{Iso}_F^{\textup{sym}}(X,X^*)$ and set $|\rho|_\mu^{\frac{1}{2}} = \mu(K) q^{\frac{k}{2}}$ for $K$ any lattice in $X$ and $|\rho|_{\mu_K}=q^k$. The scalar
$$\omega(\psi \circ Q_{\frac{1}{2} \rho}) = \frac{\Omega_\mu ( \psi \circ Q_{\frac{1}{2} \rho})}{{|\rho|}_\mu^\frac{1}{2}}.$$
is the usual Weil factor associated to $Q_{\frac{1}{2} \rho}$.
\item \label{facteur_de_weil_non_norm_hilbert_symbol_pt} 
For $a \in F^\times$, let $Q_a$ be the quadratic from $Q_a(x) = a x^2$ over $F$. Let $\mu$ be a Haar measure of $F$. The scalar
$$\Omega_{a,b} = \frac{\Omega_\mu(\psi \circ Q_a)}{\Omega_\mu( \psi \circ Q_b)} \in R^\times$$
does not depend on the choice of $\mu$. Moreover, for all $a$ and $b$ in $F^\times$, we have
$$(a,b)_F = \frac{\Omega_\mu(\psi \circ Q_1) \Omega_\mu(\psi \circ Q_{ab})}{\Omega_\mu(\psi \circ Q_a) \Omega_\mu(\psi \circ Q_b)} = \frac{\Omega_{ab,1}}{\Omega_{a,1} \Omega_{b,1}}.$$ \end{enumerate} \end{prop}

\begin{proof} a) b) c) The first two points are direct consequences of the definition of the non-normalised Weil factor. The third one comes from a change of variables $x=\phi^{-1}(x')$ which gives $\Omega_\mu(\psi \circ Q) = \Omega_{\phi \cdot \mu}(\psi \circ Q_\phi)$. When $\phi \in \textup{Aut}_F(X)$ preserves $\textup{rad}(Q)$ then $\Omega_\mu(\psi \circ Q) = \Omega_{\phi \cdot \mu}(\psi \circ Q_\phi) = |\phi| \Omega_\mu( \psi \circ Q_\phi)$ by b) and $\phi \cdot \mu = |\phi| \mu$.

d) This is a consequence of the compatibility of Haar measures with products of spaces. Here $X_1 \times X_2 = X_1 \oplus X_2$ and $C_c^\infty(X_1) \otimes C_c^\infty(X_2) \cong C_c^\infty(X_1 \oplus X_2)$ is the canonical isomorphism induced by $(f_1 \otimes f_2)(x_1 \oplus x_2) = f_1(x_1) f_2(x_2)$. Then $\mu_1 \otimes \mu_2$ is a Haar measure of $(X_1)_{Q_1} \oplus (X_2)_{Q_2} = (X_1 \oplus X_2)_{Q_1 \oplus Q_2}$ and we can decompose integrals accordingly to obtain the result.

e) This follows from the fact that, if $f \in C_c^\infty(X)$ is fixed, the map
$$\rho \mapsto \int_X \int_X f(y-x) \psi( Q_\rho (x)) d\mu (x) d\mu (y) \textup{ is locally constant}.$$

f) Note that $|\rho|_\mu^{\frac{1}{2}} = \mu(K) q^{\frac{k}{2}}$ does not depend on the choice of $K$ as $\mu_{K'} = \mu_{K'}(K) \mu_K$ and $|\rho|_{\mu_{K'}} = \mu_{K'}(K)^2 |\rho|_{\mu_K}$ and $\mu = \mu(K) \mu_K$, so $\mu(K') = \mu(K) \mu_K(K')$. To prove the link with the classical Weil factor, we use \cite[Prop 3.3]{ct} evaluated at $x^*=0$, which gives the Weil factor $\gamma$ of $Q_{\frac{1}{2} \rho}$ as the scalar satisfying 
$$\int_X \int_X f(y-x) \psi( Q_{\frac{1}{2}\rho}(x)) d\mu (x) d\mu (y) = \gamma \ |\rho|_{\mu}^{\frac{1}{2}} \int_X f(x) d \mu (x).$$
Therefore $\Omega_\mu(\psi \circ Q_{\frac{1}{2} \rho})=\gamma |\rho|_\mu^{\frac{1}{2}}$. 

g) The scalar $\Omega_{a,b}$ does not depend on $\mu$ thanks to b). We first assume that $R$ contains a square root of $q$. After fixing this square root, we can relate the non-normalised Weil factor to the usual Weil factor thanks to f) to obtain
$$\frac{\Omega_\mu(\psi \circ Q_1) \Omega_\mu(\psi \circ Q_{ab})}{\Omega_\mu(\psi \circ Q_a) \Omega_\mu(\psi \circ Q_b)} = \frac{\omega(\psi \circ Q_1) \omega(\psi \circ Q_{ab})}{\omega(\psi \circ Q_a) \omega(\psi \circ Q_b)} = (a,b)_F$$
where the last equality is a consequence of \cite[4.3]{ct}. When $R$ does not contain a square root of $q$, we can adjoin this square root to $R$ and work over an extension $R'$. The identity then holds in $R'$, with all scalars being already in $R$, so it holds in $R$ too. \end{proof}

\subsection{} \label{sec:remark_on_the_non_normalised_version} As opposed to the usual Weil factor, the non-normalised Weil factor is not necessarily trivial on split quadratic forms. It is also not invariant under isometries, though it transforms in a nice way according to c) above. Our preference for the non-normalised Weil factor comes from the fact that it is more intrinsic than the Weil factor. Indeed, we only need to be able to define $\psi$ in order to define $\Omega_\mu(\psi \circ Q)$, whereas the usual weil factor typically requires on top of that the existence of a square root of $q$. For instance, if $F=\mathbb{F}_3((t))$ and $R=\mathbb{Q}[\zeta_3]$, then $i \sqrt{3} \in R$ but $i$ and $\sqrt{3}$ are not in $R$. We can define a non-trivial character $\psi : F \to R^\times$ and the scalar $\omega(\psi \circ Q)$ can take the value $i$, but we will always have $\Omega_\mu(\psi \circ Q) \in R$. Moreover, unlike $\Omega_\mu(\psi \circ Q)$ when $R$ has positive characteristic, the definition of $\omega(\psi \circ Q)$ depends on the choice of a square root of $q$.

\subsection{} We now give a product formula for the non-normalised factor when $Q$ is realised in an orthogonal basis of $X$. We assume $Q$ is non-degenerate and $\mathcal{B}=\{v_1, \dots, v_m \}$ is an orthogonal basis for $Q$. This choice of basis induces an isomorphism from $X$ to $F^m$ and we denote it by $\phi_\mathcal{B}$. We fix a Haar measure $\mu_F$ of $F$. We form the Haar measure $\otimes \mu_F$ of $F^m$ and consider its pullback $\phi_\mathcal{B}^{-1} \cdot (\otimes \mu_F)$, which is Haar measure of $X$. We set $a_i = Q(v_i)$. The determinant $\textup{det}_\mathcal{B}(Q) = \prod a_i$ does depend on the choice of $\mathcal{B}$, whereas the Hasse invariant $h_F(Q) = \prod_{i < j} (a_i,a_j)_F$ does not. By combining the points in Proposition \ref{prop:non_normal_weil_factor}, we obtain:

\begin{cor} \label{cor:weil_factor_hasse_inv} We have 
$$\Omega_{\phi_{\mathcal{B}}^{-1} \cdot ( \otimes \mu_F )} (\psi \circ Q ) = \Omega_{\textup{det}_{\mathcal{B}}(Q),1} \times \Omega_{\mu_F}(\psi \circ Q_1)^m h_F(Q).$$
Furthermore, if $Q_{\textup{Id}_{\mathcal{B}}}$ is the non-degenerate quadratic form associated to the identity in the basis $\mathcal{B}$, we have for $\mu$ a Haar measure of $X$ that
$$\Omega_\mu(\psi \circ Q) = \Omega_{\textup{det}_{\mathcal{B}}(Q),1} \times \Omega_{\mu}(\psi \circ Q_{\textup{Id}_{\mathcal{B}}}) h_F(Q).$$ \end{cor}

\subsection{} We give an interpretation of the non-normalised Weil factor as a way to normalise the Fourier transform with respect to $\psi$. This normalisation is more natural because it only requires that the coefficient field contains the values of $\psi$, as opposed to other normalisations of the Fourier transform which may require some choices -- such as the choice of a square root of $q$.

\begin{prop} \label{prop:normalisation_Fourier_transform}  Let $\rho \in \textup{Iso}_F^{\textup{sym}}(X,X^*)$ and let $\mu$ be a Haar measure of $X$. 

\begin{enumerate}[label=\textup{\alph*)}]
\item The Haar measure
$$\mu_\rho = \Omega_\mu(\psi \circ Q_{\frac{1}{2} \rho})^{-1} \mu$$
does not depend on the choice of $\mu$.
\item If $\star_{\mu_\rho}$ denotes the convolution product in $C_c^\infty(X)$ and $\cdot$ the multiplication of functions, the Fourier transform operator
$$\mathcal{F}_{\mu_\rho} : f \in (C_c^\infty(X),\star_{\mu_\rho}) \mapsto \bigg( x \mapsto \int_X \psi( \rho(x)(u)) f(u) d \mu_\rho(u) \bigg) \in (C_c^\infty(X),\times)$$
is an isomorphism of algebras.
\item For all $f \in C_c^\infty(X)$, we have
$$\mathcal{F}_{\mu_\rho}^4 f= \varepsilon^2 f \textup{ and } \mathcal{F}_{\mu_\rho}^2 f : x \mapsto \varepsilon f(-x)$$
where $\displaystyle \varepsilon = \Omega_{-1,1}^m \big(-1,\textup{det}(Q_{\frac{1}{2} \rho})\big)_F$ and $\varepsilon^2 = \big(-1,-1\big)_F^m$.
\end{enumerate} 
\end{prop}

\begin{proof} a) Proposition 2.4 b) ensures that $\mu_\rho$ does not depend on the choice of $\mu$.

\noindent b) We refer to \cite[Prop 1.2]{ct} to show $\mathcal{F}_{\mu_\rho}$ is an isomorphism of algebras as the proof is the same.

\noindent c) Let $K$ be a compact open subgroup in $X$ and $K^\perp = \{ x \in X \ | \ \forall k \in K, \psi(\rho(x)(k))=1\}$. Let $1_K$ be the characteristic function of $K$. A routine calculation gives
$$\mathcal{F}_{\mu_\rho}^2 1_K = \mu_\rho(K) \mu_\rho (K^\perp) \times 1_K.$$
We set $\varepsilon = \mu_\rho(K) \mu_\rho (K^\perp)$. By definition
$$\varepsilon = \mu_\rho(K) \mu_\rho(K^\perp) = \Omega_\mu(\psi \circ Q_{\frac{1}{2} \rho})^{-2} \times \mu(K) \mu(K^\perp).$$
Let $K' = \{ x^* \in X^* \ | \ \forall u \in X, \psi(x^*(u))=1 \}$ in $X^*$, then
$$\mu(K) \mu(K^\perp) = \frac{\mu(\rho^{-1} K')}{\mu^*(K')} = |\rho|_\mu.$$
As $\omega(\psi \circ Q_{- \frac{1}{2} \rho}) \omega(\psi \circ Q_{\frac{1}{2} \rho}) = 1$ by \cite[Prop 3.2]{ct}, we deduce from Proposition \ref{prop:non_normal_weil_factor} f)
$$\Omega_\mu(\psi \circ Q_{- \frac{1}{2} \rho}) \Omega_\mu(\psi \circ Q_{\frac{1}{2} \rho} )= |\rho|_\mu \textup{ and } \varepsilon = \frac{\Omega_\mu(\psi \circ Q_{-\frac{1}{2} \rho} )}{\Omega_\mu(\psi \circ Q_{\frac{1}{2} \rho} )}.$$

Because the two quadratic forms $Q_{-\frac{1}{2} \rho}$ and $Q_{\frac{1}{2} \rho}$ can be put in diagonal form in the same basis $\mathcal{B}$, Corollary \ref{cor:weil_factor_hasse_inv} yields
$$\varepsilon = \frac{\Omega_{\textup{det}_{\mathcal{B}}(Q_{- \frac{1}{2}\rho}),1}}{\Omega_{\textup{det}_{\mathcal{B}}(Q_{\frac{1}{2}\rho}),1}} \times \frac{h_F(Q_{-\frac{1}{2}\rho})}{h_F( Q_{\frac{1}{2}\rho})}.$$
On the one hand, we deduce from Proposition \ref{prop:non_normal_weil_factor} g) and $\textup{det}_{\mathcal{B}}(Q_{- \frac{1}{2}\rho}) = (-1)^m \textup{det}_{\mathcal{B}}(Q_{\frac{1}{2}\rho})$, the equality
\begin{eqnarray*}
\frac{\Omega_{\textup{det}_{\mathcal{B}}(Q_{- \frac{1}{2}\rho}),1}}{\Omega_{\textup{det}_{\mathcal{B}}(Q_{\frac{1}{2}\rho}),1}} &=& \Omega_{(-1)^m,1} \times \big((-1)^m, \textup{det}_\mathcal{B}(Q_{\frac{1}{2} \rho})\big)_F \\
 &=& (\Omega_{-1,1})^m \times (-1,-1)_F^{\frac{m(m-1)}{2}} \times \big(-1, \textup{det}_\mathcal{B}(Q_{\frac{1}{2}\rho}) \big)_F^m. \end{eqnarray*}
On the other hand, as $(-a_i,-a_j)_F = (-1, - a_i a_j)_F \times (a_i, a_j)_F$, we get
\begin{eqnarray*} h_F(Q_{- \frac{1}{2} \rho}) &=& \big( -1, (-1)^{\frac{m(m-1)}{2}} \textup{det}(Q_{\frac{1}{2} \rho})^{m-1}) \big)_F \times  h_F(Q_{\frac{1}{2} \rho}) \\
&=& (-1,-1)_F^{\frac{m(m-1)}{2}} \times  \big(-1,\textup{det}(Q_{\frac{1}{2} \rho}) \big)_F^{m-1} \times h_F(Q_{\frac{1}{2}\rho}). \end{eqnarray*}
This yields the desired equality for $\varepsilon$ first, and for $\varepsilon^2$ then by Proposition \ref{prop:non_normal_weil_factor} g).

Regarding the powers of $\mathcal{F}_{\mu_\rho}$, a classical argument consists in first proving it for characteristic functions of the form $1_{x+K}$ and deduce it for all functions in $C_c^\infty(X)$. \end{proof}

The previous proposition justifies to make the following definition:

\begin{defi} We call $\mathcal{F}_{\mu_\rho}$ the Fourier transform. \end{defi}

\subsection{} It is not necessarily possible to normalise the Fourier transform in the usual way with the sole values of $\psi$, \textit{i.e.} such that it satisfies $\mathcal{F}^2 f(x) = f(-x)$. Let $F=\mathbb{F}_3((t))$ and $R=\mathbb{Q}[\zeta_3]$. We have $\varepsilon = -1$ when $|\rho|_\mu = 3$. The classical normalisation requires to divide by $\sqrt{3}$, which does not belong to $R$. Moreover, when $R$ has positive characteristic, there is no canonical choice of a square root of $q$, whereas $\mathcal{F}_{\mu_\rho}$ is defined independently of any choice. This example echoes Section \ref{sec:remark_on_the_non_normalised_version} in the context of Fourier transforms.

\section{The metaplectic cocycle} \label{sec:the_met_cocycle}

\subsection{} We recall some notations from \cite{rao} and \cite{kudla}. Let $W=X \oplus Y$ be a complete polarisation and fix a basis $\{e_1, \dots, e_m\}$ of $X$. This determines a dual basis $\{f_1, \dots, f_m\}$ in $Y$ \textit{i.e.} $\langle e_i , f_j \rangle = \delta_{i,j}$ for all $i, j$. For $S \subseteq \{1, \dots, m\}$, let $X_S$ be the subspace of $X$ generated by $(e_i)_{i \in S}$. If ${}^c S$ denotes the complement of $S$, then $X_{{}^c S}$ is a complement of $X_S$ in $X$. We use similar notations for $Y$. Then $W_S = X_S \oplus Y_S$ is a symplectic subpace of $W$ with orthogonal complement $W_{{}^c S}$. Let $w_S \in \textup{Sp}(W)$ be defined by
$$w_S(e_i) = \left\{ \begin{array}{lc} f_i & \textup{if } i \in S\\ e_i & \textup{if } i \notin S \end{array} \right.  \textup{ and } w_S(f_i) = \left\{ \begin{array}{cc} - e_i & \textup{if } i \in S\\ f_i & \textup{if } i \notin S \end{array} \right. .$$
For all $0 \leq j \leq m$, we set $w_j = w_{\{1, \dots, j\}}$ and $\Omega_j = P(X) w_j P(X)$. The Bruhat decomposition reads
$$\textup{Sp}(W) = \coprod_j \Omega_j.$$

\subsection{} Let $g \in \Omega_j$. Let $p_1$ and $p_2$ be in $P(X)$ such that $g=p_1 w_j p_2$. Denote by $\phi_1$ the isomorphism between $gX \cap X \backslash X$ and $w_j X \cap X \backslash X$ induced by
$$x \in X \mapsto \overline{p_1^{-1} x} \in w_j X \cap X \backslash X.$$
Let $Q_j$ be the non-degenerate quadratic form on $w_j X \cap X \backslash X$ defined by
$$Q_j(x) = \frac{1}{2} \langle w_j x , x \rangle.$$
For all Haar measures $\mu$ on $w_j X \cap X \backslash X$, we denote by $\mu_{w_j} = \Omega_\mu(\psi \circ Q_j)^{-1} \mu$ the Haar measure of Proposition \ref{prop:normalisation_Fourier_transform} which normalises the Fourier transform with respect to $\psi$ and the symmetric isomorphism $\rho : x \mapsto \langle w_j x , - \rangle$.

\begin{lem} The Haar measure on $g X \cap X \backslash X$ defined by
$$\mu_g = \Omega_{1,\textup{det}_X(p_1 p_2)} \times \phi_1 \cdot \mu_{w_j}$$
does not depend on the choice of $p_1$ and $p_2$. \end{lem}

\begin{proof} Let $g = p_1 w_j p_2 = p_1' w_j p_2'$ be two decompositions of $g \in \Omega_j$. Then $\phi_1^{-1} \circ \phi_1'$ is an automorphism of $w_j X \cap X \backslash X$. We need to check that
$$\Omega_{1,\textup{det}_X(p_1 p_2)} \times |\textup{det}(\phi_1^{-1} \circ \phi_1')|_F = \Omega_{1,\textup{det}_X(p_1' p_2')}.$$
However, according to \cite[Lem. 3.4]{rao}, we have
$$\det(\phi_1^{-1} \circ \phi_1')^2 = \textup{det}_X ( p_1^{-1} p_1' p_2' p_2^{-1})$$
because $p_1^{-1} p_1' w_j p_2' p_2^{-1} = w_j$. We now apply Proposition \ref{prop:non_normal_weil_factor} to obtain
$$\Omega_{1,\textup{det}(\phi_1^{-1} \circ \phi_1')^2 \textup{det}_X(p_1 p_2)} = |\textup{det} (\phi_1^{-1} \circ \phi_1')|_F \times \Omega_{1,\textup{det}_X(p_1 p_2)}.$$
Therefore $\mu_g$ does not depend on the choice of $p_1$ and $p_2$. \end{proof}

\subsection{} Let $g \in \Omega_j$. Let $p_1$ and $p_2$ be in $P(X)$ such that $g=p_1 w_j p_2$. We define $x(g)$ as the image of $\textup{det}_X(p_1 p_2)$ in $F^\times / F^{\times 2}$. Then $x(g)$ is well-defined and does not depend on the choice of $p_1$ and $p_2$ \cite[Lem 3.4]{rao}. Moreover, we have $x(w_S) = 1$ for all $S \subseteq \{ 1 , \dots, m \}$. By Proposition \ref{prop:non_normal_weil_factor} g), we easily obtain:

\begin{cor} \label{cor:Haar_measure_mu_g_compat_to_P(X)} For all $g \in \textup{Sp}(W)$ and all $p \in P(X)$, we have
$$\mu_{g p } = (x(p),x(g))_F \times \Omega_{1,\textup{det}_X(p)} \times \mu_g$$
and
$$\mu_{pg} = (x(p),x(g))_F \times \Omega_{1,\textup{det}_X(p)} \times \phi_p \cdot \mu_g$$
where $\phi_p : x \in gX \cap X \backslash X \mapsto \overline{p x} \in pgX \cap X \backslash X$. \end{cor}

\subsection{} Recall that $p_{S_X} : \widetilde{\textup{Sp}}_{\psi,S_X}^R(W) \to \textup{Sp}(W)$ is the canonical projection associated to the Schr\"odinger model $S_{\psi,X}$ and that the intertwining operators $I_{A_1,A_2,\mu,\omega}$ are the change of models. We define a section of $p_{S_X}$ thanks to the previous measures via
$$\sigma : g \in \textup{Sp}(W) \mapsto \sigma(g)=I_{gX,X,\mu_g,0} \circ I_g \in \widetilde{\textup{Sp}}_{\psi,S_X}^R(W).$$
We denote by $\hat{c}$ the $2$-cocyle associated to this section \textit{i.e.}
$$\hat{c} : (g,g') \in \textup{Sp}(W) \times \textup{Sp}(W) \mapsto \sigma(g) \sigma(g') \sigma(gg')^{-1} \in R^\times.$$

\begin{lem} \label{lem:metaplectic_cocycle_specific_elements} \ \begin{enumerate}[label=\textup{\alph*)}]
\item For all $g \in \textup{Sp}(W)$ and all $p \in P(X)$ 
$$\hat{c}(g,p) = \hat{c}(p,g) = \hat{c}(p^{-1},g)= (x(p),x(g))_F.$$
\item For $S$ and $S'$ in $\{1,\dots, m \}$, we set $l = | S \cap S'|$. Then
$$\hat{c}(w_S,w_{S'}) = (-1,-1)_F^{\frac{l(l+1)}{2}}.$$
\item Let $S \subseteq \{1, \dots, m\}$ and $W = W_S + W_{{}^c S}$, so that the subgroup of $\textup{Sp}(W)$
$$G_S = \{ g \in \textup{Sp}(W) \ | \ g(W_S) \subset W_S \textup{ and } g|_{W_S} = \textup{Id}_{W_S} \}$$ 
is canonically isomorphic to $\textup{Sp}(W_{{}^c S})$ via the restriction to $W_{{}^c S}$. We use similar notations for ${}^c S$. Then for all $g \in G_S$ and all $g' \in G_{{}^c S}$, we have
$$\hat{c}(g,g') = \hat{c}(g',g) = (x(g) , x(g'))_F.$$ \end{enumerate}  \end{lem}

\begin{proof} a) As $\sigma(p) = \Omega_{1,\textup{det}_X(p)}\times I_p$, we deduce from Corollary \ref{cor:Haar_measure_mu_g_compat_to_P(X)} that
$$\sigma(pg) = (x(p) , x(g))_F \times \sigma(g) \sigma(p) \textup{ and } \sigma(pg) = (x(p),x(g))_F \times \sigma(p) \sigma(g).$$
Therefore $\hat{c}(p,g)=\hat{c}(g,p) = (x(p),x(g))_F$. Moreover $x(p) = x(p^{-1})$ by definition.

\noindent b) Denote by $\gamma_S$ the isomorphism $X_S \simeq Y_S$ induced by $w_S$. For $f \in C_c^\infty(Y)$, we have\
$$\sigma(w_S) f : y \mapsto \int_{w_S X \cap X \backslash X} \psi(\langle a,y \rangle) f((-\gamma_S a,0)) d \mu_{w_S}(a).$$
Set $\rho_S : x \in X_S \mapsto \langle \gamma_S x , - \rangle \in X_S^*$ which is symmetric. We define
$$M_{\gamma_S} : f \in C_c^\infty(Y_S) \mapsto \mathcal{F}_{\mu_{\rho_S}} (f \circ (-\gamma_S) ) \circ (-\gamma_S)^{-1} \in C_c^\infty(Y_S)$$
where $\mathcal{F}_{\mu_{\rho_S}}$ is the Fourier transform in Proposition \ref{prop:normalisation_Fourier_transform} \textit{i.e.} for $f' \in C_c^\infty(X_S)$ and $x \in X_S$
$$\mathcal{F}_{\mu_{\rho_S}} f' (x) =\int_{X_S} \psi( \langle \gamma_S x, a \rangle) f'(a) d \mu_{w_S}(a).$$
Through the decomposition $C_c^\infty(Y) = C_c^\infty(Y_S) \otimes C_c^\infty(Y_{{}^c S})$, we get $\sigma(w_S) = M_{\gamma_S} \otimes \textup{Id}$. The same holds for $w_{S'}$ and $\sigma(w_{S'})$ of course.

Now, via the decomposition $C_c^\infty(Y) = C_c^\infty(Y_{S \cap S'}) \otimes C_c^\infty(Y_{S \Delta S'}) \otimes C_c^\infty(Y_{{}^c (S \cup S')})$, the operator 
$\sigma(w_S) \circ \sigma(w_{S'})$ can be written as $(M_{\gamma_S} \circ M_{\gamma_{S'}}) \otimes M_{\gamma_{S \Delta S'}} \otimes \textup{Id}$ where $\gamma_{S \Delta S'}$ is associated to $w_{S \Delta S'}$. However, the restriction of $M_{\gamma_S}$ and $M_{\gamma_{S'}}$ to $C_c^\infty(Y_{S \cap S'})$ are both equal to $M_{\gamma_{S \cap S'}}$. Furthemore
\begin{eqnarray*} M_{\gamma_{S \cap S'}}^2 f &=& \mathcal{F}_{\mu_{\rho_{ S\cap S'}}} (M_{\gamma_{S \cap S'}} f \circ (-\gamma_{S \cap S'})) \circ (-\gamma_{S\cap S'})^{-1} \\
&=& \mathcal{F}_{\mu_{\rho_{ S\cap S'}}} \big( \mathcal{F}_{\mu_{\rho_{ S\cap S'}}}  (f \circ (-\gamma_{S \cap S'})) \big) \circ (-\gamma_{S\cap S'})^{-1} \\
&=& \mathcal{F}_{\mu_{\rho_{S \cap S'}}}^2 (f \circ (-\gamma_{S \cap S'}) ) \circ (-\gamma_{S\cap S'})^{-1}. \end{eqnarray*}
We finally obtain thanks to Proposition \ref{prop:normalisation_Fourier_transform} that for all $f \in C_c^\infty(Y_{S \cap S'})$ and all $y \in Y_{S \cap S'}$
$$M_{\gamma_{S \cap S'}}^2 f (y )= (-1, \textup{det}(Q_{\frac{1}{2} \rho_{S \cap S'}}))_F (\Omega_{-1,1})^l \times  f(-y).$$

Since $w_S w_{S'}=  w_{S \Delta S'} a_{S \cap S'}$ where $a_{S \cap S'} =(-\textup{Id}_{W_{S \cap S'}}) \oplus \textup{Id}_{W_{{}^c (S \cap S')}}$ is in $P(X)$, the measure $\mu_{w_S w_{S'}} = \Omega_{1, (-1)^l} \times \mu_{w_{S \cap S'}}$. In the previous decomposition, $\sigma(w_S w_{S'})$ can be written as $A_{S \cap S'} \otimes M_{\gamma_{S \Delta S'}} \otimes \textup{Id}$ where $A_{S \cap S'} f(y) = \Omega_{1, (-1)^l} \times f(-y)$. Therefore, because
$$(-1,\textup{det}(Q_{\frac{1}{2} \rho_{S \cap S'}}))_F = (-1,2^{-l})_F = (-1 , 2^l)_F = (-1,2)_F^l  = 1,$$
we obtain
$$ \hat{c}(w_S,w_{S'}) = \Omega_{-1,1}^l \times \Omega_{(-1)^l,1}. $$
By applying Proposition \ref{prop:non_normal_weil_factor} g) repeatedly, we get
$$\Omega_{(-1)^l,1} = (\Omega_{-1,1})^l \times (-1,-1)_F^{\frac{l(l-1)}{2}}.$$
Finally, because $(\Omega_{-1,1})^2 = (-1,-1)_F$, we conclude that $\hat{c}(w_S,w_{S'}) = (-1,-1)_F^{\frac{l(l+1)}{2}}$.

\noindent c) On the one hand, for all $f \in S_X$, an explicit computation of $\sigma(g) \circ \sigma(g') f ((0,0))$ gives
$$\int_{gX \cap X \backslash X} \int_{g'X \cap X \backslash X} f(((g')^{-1} a' , 0 )  ((g')^{-1} g^{-1} a , 0 ) ) d \mu_{g'}(a') d \mu_g(a).$$
On the other hand, the morphism
$$x \in X \mapsto (p_g(x),p_{g'}(x)) \in (g' X \cap X \backslash X ) \times ( gX \cap X \backslash X )$$
has kernel $g g' X \cap X$. However, $g$ and $g'$ commute and each one of them induces the identity map by passing to the respective quotients $g'X \cap X \backslash X$ and $g X \cap X \backslash X$. We use Proposition \ref{prop:non_normal_weil_factor} g) again to obtain
$$\mu_g \otimes \mu_{g'} = (x(g),x(g'))_F \mu_{g g'}.$$
By a change of variables, we get
$$\sigma(g) \circ \sigma(g') f ((0,0)) = (x(g),x(g'))_F \times \sigma(g g') f ((0,0))$$
for all $f \in S_X$. Therefore $\sigma(g) \circ \sigma(g') = (x(g),x(g'))_F \times \sigma(g g')$. \end{proof}

\begin{defi} We let $\textup{Sp}(W) \times_{\hat{c}} R^\times$ be the set $\textup{Sp}(W) \times R^\times$ with group law
$$(g, \lambda) \cdot (g', \lambda') = (g g' , \hat{c}(g,g') \lambda \lambda').$$ \end{defi}

\begin{theo} We follow the same case distinctions as Theorem \ref{thm:metaplectic_group}.
\begin{enumerate}[label=\textup{\alph*)}]
\item If $F$ is finite, or if the characteristic $\ell$ of $R$ is $2$, then $\hat{c}$ is the trivial $2$-cocyle. In particular $\sigma$ is a section of $p_{S_X}$ that is a group morphism. Furthermore, such a group morphism is unique, except in the exceptional case $F= \mathbb{F}_3$ and $\textup{dim}_F W= 2$. The group morphism $\sigma$ always defines an isomorphism of central extensions
$$(g, \lambda) \in \textup{Sp}(W) \times R^\times \mapsto (g,\lambda \sigma(g)) \in \widetilde{\textup{Sp}}_{\psi,S_X}^R(W).$$
\item If $F$ is local non-archimedean and $\ell \neq 2$, then $\hat{c}$ takes value in $\{ \pm 1 \}$. In particular $\sigma$ is the unique section of $p_S$ realising an isomorphism of central extensions
$$(g, \lambda) \in \textup{Sp}(W) \times_{\hat{c}} R^\times \mapsto (g, \lambda \sigma(g)) \in \widetilde{\textup{Sp}}_{\psi,S_X}^R(W).$$
Its restriction to $\textup{Sp}(W) \times_{\hat{c}} \{ \pm 1 \}$ induces an isomorphism with $\widehat{\textup{Sp}}_{\psi,S_X}^R(W)$. \end{enumerate} \end{theo}

\begin{proof} We treat both points at the same time. Let $p_1, p_2, p \in P(X)$ and $g_1, g_2 \in \textup{Sp}(W)$. By decomposing $\sigma(g_1 p^{-1}) \circ \sigma(p g_2)$, we deduce from Lemma \ref{lem:metaplectic_cocycle_specific_elements} that
$$\hat{c} (g_1 p^{-1} , p g_2) =  \hat{c}(g_1,p) \hat{c}(p,g_2) \hat{c}(p,p) \hat{c}(g_1 , g_2),$$
and by decomposing $\sigma(p_1 g_1) \circ \sigma(g_2 p_2)$ as well, we obtain
$$\hat{c}(p_1g_1,g_2p_2)=\hat{c}(p_1 ,g_1) \hat{c}(g_2,p_2) \hat{c}(g_1,g_2) \hat{c}(p_1,g_1g_2) \hat{c}(p_1 g_1 g_2 ,p_2).$$
Therefore by Lemma \ref{lem:metaplectic_cocycle_specific_elements} a), we get
$$\hat{c}(p_1 g_1 p^{-1} , p g_2 p_2) \hat{c} (g_1,g_2)^{-1} = \left\{ \begin{array}{cl} 1 & \textup{ in case a);} \\
\pm 1 & \textup{ in case b)}. \end{array} \right.$$

There remains to prove that, for $g_1$ and $g_2$ well-chosen, the $2$-cocycle $\hat{c}(g_1,g_2)$ is either trivial or has values in $\{ \pm 1\}$. To do so, we use a decomposition associated to the Leray invariant \cite[Th 2.16]{rao}. For all $g_1$ and $g_2$ in $\textup{Sp}(W)$, there exist $S_1$, $S_2$ and $S$ in $\{1, \dots ,m\}$, and an antisymmetric isomorphism $\rho : Y_S \to X_S$, \textit{i.e.} $\rho^* = -\rho$, such that $(S_1 \cup S_2) \cap S = \emptyset$ and $g_1 = p_1 w_{S \cup S_1} u_{\rho} p^{-1}$ and $g_2 = p w_{S \cup S_2} p_2$. We simply need to compute $\hat{c}(w_{S \cup S_1} u_\rho, w_{S\cup S_2})$ for the previous morphism $\rho$. However, by computing $\sigma(w_{S \cup S_1} u_\rho) \circ \sigma(w_{S\cup S_2})$ in different ways, we obtain by Lemma \ref{lem:metaplectic_cocycle_specific_elements} that
\begin{eqnarray*} \hat{c}(w_{S \cup S_1} u_\rho, w_{S\cup S_2}) & = & \hat{c}(w_S u_\rho, w_S) \hat{c}(w_{S_1} w_{S_2} , w_S u_\rho w_S)\hat{c}(w_{S_1} ,w_{S_2}) \\
& = & \hat{c}(w_S u_\rho, w_S) \ ((-1)^l , x(w_S u_\rho w_S))_F \ (-1,-1)_F^{\frac{l(l+1)}{2}}.
\end{eqnarray*}
where $l = |S_1 \cap S_2|$.

There only remains to study $\hat{c}(w_S u_\rho, w_S)$ to finish the proof:

\begin{lem} \label{lem:cocycle_formula_w_s_u_rho_w_s} We have
$$\hat{c}(w_S u_\rho,w_S)=(-2, \textup{det}(Q_{\gamma_S \rho \gamma_S}))_F \times h_F(Q_{\gamma_S \rho \gamma_S})$$
where $Q_{\gamma_S \rho \gamma_S}(x)= \langle x , \gamma_S \rho \gamma_S x \rangle$ is a non-degenerate quadratic form over $X_S$. \end{lem}

\begin{proof} For all $S \subseteq \{1,\dots,m\}$ and all $\rho : Y_S \to X_S$ such that $u_\rho \in \textup{Sp}(W)$, we want to compute the composition $\sigma(w_S u_\rho) \circ \sigma(w_S)$ in terms of $\sigma(w_S u_\rho w_S)$. Let $f \in S_X$. As $\sigma(w_S u_\rho) = \sigma(w_S) \circ  \sigma(u_\rho)$ thanks to Lemma \ref{lem:metaplectic_cocycle_specific_elements}, we have
$$\sigma(w_S u_\rho) \circ \sigma(w_S) f ((0,0)) = \int_{X_S} (\sigma(u_\rho) \circ \sigma(w_S) f)((w_S^{-1} a,0)) d \mu_{w_S}(a).$$
However $(\sigma(u_\rho) \circ \sigma(w_S) f)((w_S^{-1} a,0)) = \psi(\frac{1}{2} \langle w_S^{-1} a , (-\rho) w_S^{-1} a \rangle) \times ( \sigma(w_S) f ) ((w_S^{-1} a, 0 ))$ by the formulas of the Schr\"odinger model. Moreover
\begin{eqnarray*} \sigma(w_S) f ((w_S^{-1} a , 0 )) &=& \int_{X_S} f((w_S^{-1} a',0)(w_S^{-2} a , 0 ) d \mu_{w_S}(a') \\
&=& \int_{X_S} \psi( \langle a', w_S^{-1} a \rangle) f((w_S^{-1}a',0))) d \mu_{w_S}(a') \\
&=& |\phi_{\rho,S}|^{-1} \int_{X_S} \psi( \langle w_S^{-1}a, \rho w_S^{-1} a'' \rangle) f((- w_S^{-1} \rho w_S^{-1} a'',0) d \mu_{w_S}(a'')
\end{eqnarray*}
by the change of variables $a'=\phi_{\rho,S} (a'') = -\rho w_S^{-1} a''$ for the automorphism $\phi_{\rho,S}$ of $X_S$ induced by $-\rho w_S^{-1}$ over $X_S$.

Since $u_\rho^{-1} = u_{- \rho}$, we get
$$- w_S^{-1} \rho w_S^{-1} a'' = w_S^{-1} u_\rho^{-1} w_S^{-1} a'' - w_S^{-2} a''.$$
Therefore
$$f((-w_S^{-1} \rho w_S^{-1} a'',0))= \psi(\frac{1}{2} \langle w_S^{-1} a'' , (-\rho) w_S^{-1} a'' \rangle) f((w_S^{-1} u_\rho^{-1} w_S^{-1} a'' , 0)).$$
This leads to $\sigma(w_S u_\rho) \circ \sigma(w_S) f ((0,0))$ up to a factor $|\phi_{\rho,S}|^{-1}$ \textit{i.e.}
$$\int_{X_S} \int_{X_S} \psi \bigg(\frac{1}{2} \langle w_S^{-1} (a-a'') , (-\rho) w_S^{-1} (a - a'') \rangle \bigg) \ (I_{w_S u_\rho w_S}f)((a'',0)) d\mu_{w_S}(a'') d \mu_{w_S}(a).$$
Thanks to the non-normalised Weil factor, we can simplify the latter as
$$\Omega_{\mu_{w_S}}(\psi \circ Q_S) \times \int_{X_S} (I_{w_S u_\rho w_S} f)((a'',0)) d\mu_{w_S}(a'')$$
where $Q_S(x) = - \frac{1}{2} \langle x , w_S \rho w_S^{-1} x \rangle$.

Furthermore $w_S u_\rho w_S \in G_{{}^c S}$ can be decomposed in $W_S = X_S + Y_S$ as
$$w_S u_\rho w_S =  \left[ \begin{array}{cc} * & * \\
\gamma_S \rho \gamma_S & * \end{array} \right] \textup{ where } w_S = \left[ \begin{array}{cc} * & * \\
\gamma_S & * \end{array} \right] \textup{ and } \gamma_S^*=-\gamma_S.$$
Since $\gamma_S \rho \gamma_S$ has rank $|S|$, there exist $p_1$ and $p_2$ in $P(X_S)$ such that $w_S u_\rho w_S = p_1 w_S p_2$. In addition, there exists a decomposition of the form
$$w_S u_\rho w_S =  \left[ \begin{array}{cc} \textup{Id}_{X_S} & * \\
0 & \textup{Id}_{Y_S} \end{array} \right] w_S \left[ \begin{array}{cc} a & * \\
0 & (a^*)^{-1} \end{array} \right] .$$
In particular, such a decomposition imposes $\gamma_S a= \gamma_S \rho \gamma_S$ \textit{i.e.} $a = \rho \gamma_S \in \textup{GL}_F(X_S)$. With these notations, we have $\phi_{\rho,S} = \rho \gamma_S$ and $Q_S(x) =  \frac{1}{2} \langle x , \gamma_S \rho \gamma_S x \rangle$. 

The expression of the measure $\mu_{w_S u_\rho w_S}$ then becomes
$$\mu_{w_S u_\rho w_S} = \Omega_{1, \textup{det}(\rho \gamma_S)} \times  \mu_{w_S}.$$
This leads to the formula
$$\hat{c}(w_S u_\rho , w_S) = |\phi_{\rho,S}|^{-1} \times \Omega_{\mu_{w_S}}(\psi \circ Q_S) \times  \Omega_{\textup{det}(\rho \gamma_S),1}.$$
We are going to simplify it in what follows.

On the one hand, in the standard basis $\mathcal{B}$ of $X_S$, Corollary \ref{cor:weil_factor_hasse_inv} gives
$$\Omega_{\mu_{w_S}} ( \psi \circ Q_{\frac{1}{2} \gamma_S \rho \gamma_S}) = \Omega_{\textup{det}_\mathcal{B}(Q_{\frac{1}{2} \gamma_S \rho \gamma_S}),1} \times h_F(Q_{\frac{1}{2} \gamma_S \rho \gamma_S}) \times \Omega_{\mu_{w_S}} (\psi \circ Q_{\gamma_S}).$$
But $\Omega_{\mu_{w_S}}(\psi \circ Q_{\gamma_S}) = (\Omega_{1,\frac{1}{2}})^{|S|} \times \Omega_{\mu_{w_S}}(\psi \circ Q_{\frac{1}{2}\gamma_S}) = (\Omega_{1,\frac{1}{2}})^{|S|}$, where the last equality can be deduced form the definition of $\mu_{w_S}$. Furthermore 
$$h_F(Q_{\frac{1}{2} \gamma_S \rho \gamma_S}) = (2,\textup{det}(Q_{\gamma_S \rho \gamma_S})^{|S|-1})_F \times h_F(Q_{\gamma_S \rho \gamma_S}).$$
and
$$\Omega_{\textup{det}_\mathcal{B}(Q_{\frac{1}{2} \gamma_S \rho \gamma_S}),1} = (2^{-|S|},\textup{det}(Q_{\gamma_S \rho \gamma_S}))_F \times \Omega_{2^{-|S|},1} \times \Omega_{\textup{det}_\mathcal{B}(Q_{\gamma_S \rho \gamma_S}), 1}$$
By noticing that $\Omega_{2^{-|S|},1} = (\Omega_{\frac{1}{2},1})^{|S|}$, we get
$$\Omega_{\mu_{w_S}} ( \psi \circ Q_{\frac{1}{2} \gamma_S \rho \gamma_S}) = (2, \textup{det}(Q_{\gamma_S \rho \gamma_S}))_F \times \Omega_{\textup{det}_\mathcal{B}(Q_{\gamma_S \rho \gamma_S}),1}  h_F(Q_{\gamma_S \rho \gamma_S}).$$

On the other hand, the matrix representation of the quadratic form $Q_{\gamma_S \rho \gamma_S}$ in the basis $\mathcal{B}$ is $Q_{\gamma_S \rho \gamma_S}(x) = {}^t X M X$ where $M = \textup{Mat}_\mathcal{B}(\rho \gamma_S)$ and $ X \in F^{|S|}$ is the coordinate vector associated to $x \in X$ in the basis $\mathcal{B}$. This means $\textup{det}_\mathcal{B}(Q_{\gamma_S \rho \gamma_S}) = \textup{det}(\rho \gamma_S)$. Because for all $a \in F^\times$, we have $(\Omega_{a,1})^2 = |a| \times (a,a)_F = |a| \times (-1,a)_F$, we obtain
$$\hat{c}(w_S u_\rho,w_S) = (-2,\textup{det}(Q_{\gamma_S \rho \gamma_S}) )_F \times h_F (Q_{\gamma_S \rho \gamma_S}).$$ \end{proof}

To finish the proof of a), the uniqueness outside the exceptional case is a consequence of Theorem \ref{thm:metaplectic_group}, whereas the uniqueness in b) is a classical fact about isomorphisms of central extensions, which are parametrised by characters of $\textup{Sp}(W)$ in our situation. Since the symplectic group is perfect, there is only one such isomorphism. \end{proof}

\begin{rem} The formula we obtained in the proof for $\hat{c}(w_S u_\rho, w_S)$ is slightly different from \cite{rao}, but the two $2$-cocyles are cohomologous. Indeed, these cocycles correspond to different choices of Haar measures, namely 
$$\mu_{g,\textup{Rao}} = \Omega_{\frac{1}{2}, \frac{1}{2} \textup{det}_X(p_1 p_2)} \times \phi_1 \cdot \mu_{w_j} = (2, x(g))_F \times \mu_g$$
\textit{i.e.} $\sigma_{\textup{Rao}}(g) = (2,x(g))_F \times \sigma(g)$. More generally, if $\mu_{g,\alpha} = \Omega_{\alpha, \alpha \textup{det}_X(p_1 p_2)} \times \phi_1 \cdot \mu_{w_j}$ for $\alpha \in F^\times$, this gives a $2$-cocycle $\hat{c}_\alpha$ in the same cohomology class as $\hat{c}$. \end{rem}

\subsection{} By writing $\sigma(p_1 g_1) \sigma(g_2 p_2)$ in different ways, we obtain
$$\hat{c}(p_1 g_1, g_2 p_2) = \hat{c}(p_1,g_1) \hat{c}(g_2,p_2) \hat{c}(g_1 g_2 , p_2) \hat{c}(p_1 , g_1 g_2 p_2) \hat{c}(g_1,g_2)$$
and likewise for $\sigma(g_1 p) \sigma(p^{-1} g_2)$ to obtain
$$\hat{c}(g_1 p^{-1} , p g_2) = \hat{c}(g_1,p) \hat{c}(p,g_2) \hat{c}(p,p g_2) \hat{c}(g_1,g_2).$$
Combining these facts with Lemma \ref{lem:metaplectic_cocycle_specific_elements} and Lemma \ref{lem:cocycle_formula_w_s_u_rho_w_s}, we deduce the general formula:

\begin{cor} Let $g_1$ and $g_2$ be in $\textup{Sp}(W)$. By definition of the Leray invariant, there exist $p_1, p_2, p \in P(X)$, $S \subseteq \{1,\dots,m\}$, an antisymmetric isomorphism $\rho : Y_S \to X_S$ and $S_1,S_2 \subset {}^c S$ such that $g_1 = p_1 w_{S \cup S_1} u_\rho p^{-1}$ and $g_2 = p w_{S \cup S_2} p_2$. With these notations, and by setting $l= |S_1 \cap S_2|$, we have
\begin{multline*} \hat{c}(g_1,g_2) = (x(g_1),x(g_2))_F \times  (x(g_1)x(g_2),-x(g_1 g_2))_F \times (-1,-1)_F^{\frac{l(l+1)}{2}} \\
\times ((-1)^l,x(w_S u_\rho w_S))_F \times \hat{c}(w_S u_\rho , w_S).
\end{multline*} \end{cor}

\section{Around a modular theta correspondence} \label{sec:around_modular_theta}

\subsection{} Let $(H_1,H_2)$ be a reductive dual pair in $\textup{Sp}(W)$. Recall that $H_1$ and $H_2$ are two reductive subgroups of $\textup{Sp}(W)$ that are mutual centralisers. See \cite[Chap I, 1.17]{mvw} for more details about dual pairs and their classification.

Let $p_S : \widetilde{\textup{Sp}}_{\psi,S}^R(W) \to \textup{Sp}(W)$ be the projection associated to a model $S$ of the Weil representation. We set
$$\widetilde{H}_{1,S} = p_S^{-1} (H_1) \textup{ and } \widetilde{H}_{2,S} = p_S^{-1}(H_2).$$

\begin{prop} The groups $\widetilde{H}_{1,S}$ and $\widetilde{H}_{2,S}$ are mutual centralisers in $\widetilde{\textup{Sp}}_{\psi,S}^R(W)$. \end{prop}

\begin{proof} First of all, Lemma \ref{lem:centraliser_metaplectic} ensures the centraliser of $\widetilde{H}_{1,S}$ contains $\widetilde{H}_{2,S}$. If $\tilde{g}$ is in the centraliser of $\widetilde{H}_{1,S}$, then $p_S(\tilde{g})$ is in the centraliser of $H_1$ \textit{i.e.} $p_S(\tilde{g}) \in H_2$. Hence the centraliser of $\widetilde{H}_{1,S}$ is $\widetilde{H}_{2,S}$. Similarly $\widetilde{H}_{1,S}$ is the centraliser of $\widetilde{H}_{2,S}$. \end{proof}

We then say $(\widetilde{H}_{1,S},\widetilde{H}_{2,S})$ forms a reductive dual pair in $\widetilde{\textup{Sp}}_{\psi,S}^R(W)$. We still denote  by $\omega_{\psi,S}$ the pullback of the Weil representation via the morphism given by multiplication
$$\widetilde{H}_{1,S} \times \widetilde{H}_{2,S} \to \widetilde{\textup{Sp}}_{\psi,S}^R(W).$$

\subsection{} We want to consider only a certain kind of smooth representations of these central extensions of $H_1$ and $H_2$, namely the genuine representations, as $\omega_{\psi,S}$ itself is genuine. Let $H$ be a closed subgroup $H$ of $\textup{Sp}(W)$. The category of genuine representations is
$$\textup{Rep}_R^{\textup{gen}}(\widetilde{H}_S) = \{ (\pi,V) \in \textup{Rep}_R(\widetilde{H}_S) \ | \  \forall \lambda \in R^\times, \ \pi((\textup{Id}_W,\lambda \textup{Id}_S)) = \lambda \textup{Id}_V \}.$$
Let $\textup{Irr}_R^{\textup{gen}}(\widetilde{H}_{1,S})$ be the isomorphism classes of irreducible genuine representations. Note that the contragredient of a genuine representation is not genuine. However, there are two ways to solve this issue, as Proposition \ref{prop:weil_rep_contragredient} illustrates. We can twist by a character to make the naive contragredient genuine, which requires carrying cumbersome notations due to this twist. As an alternative, we can work in the category $\textup{Rep}_R^{\textup{gen}}(\widehat{H}_S)$ for the two-fold cover of $H$ as the contragredient now remains genuine. In addition, the categories $\textup{Rep}_R^{\textup{gen}}(\widehat{H}_S)$ and $\textup{Rep}_R^{\textup{gen}}(\widetilde{H}_S)$ are equivalent. We prefer this second more elegant solution.

\subsection{} Let $\pi_1 \in \textup{Rep}_R^{\textup{gen}}(\hat{H}_1)$ be irreducible. For $V \in \textup{Rep}_R^{\textup{gen}}(\hat{H}_1 \times \hat{H}_2)$, the largest $\pi_1$-isotypic quotient $V_{\pi_1}$ of $\pi_1$ is the largest quotient on $V$ on which the action of $\hat{H}_1$ is $\pi_1$-isotypic. An alternative definition is that it is the unique quotient of $V$ which factors all maps $V \to \pi_1$. Moreover, it is endowed with an action of $\hat{H}_2$.

We now assume $R$ is an algebraically closed field. From \cite{trias_rational_weil}, we obtain: 

\begin{theo} \label{thm:Theta_pi_1_def} Let $\pi_1 \in \textup{Irr}_R^{\textup{gen}} (\widehat{H}_{1,S})$. There exists $\Theta(\pi_1) \in \textup{Rep}_R^{\textup{gen}}(\widehat{H}_{2,S})$, unique up to isomorphism, such that $(\omega_{\psi,S})_{\pi_1} \simeq \pi_1 \otimes_R \Theta(\pi_1)$. \end{theo}

\subsection{} We discuss some statements on $\Theta(\pi_1)$ at the heart of the local theta correspondence. We now suppose $F$ is local non-archimedean. Let $\pi_1 \in \textup{Irr}_R^{\textup{gen}} (\widehat{H}_{1,S})$. We consider the following first statement:
\begin{itemize}
\item[(Fin)] $\Theta(\pi_1)$ has finite length.
\end{itemize}
If (Fin) holds, the maximal semisimple quotient $\theta(\pi_1)$ of $\Theta(\pi_1)$, also called the cosocle, is well-defined. We add the second statement
\begin{itemize}
\item[(Irr)] $\theta(\pi_1)$ is irreducible or zero.
\end{itemize}
We also add when (Fin) and (Irr) hold for all $\pi_1$, the third statement
\begin{itemize}
\item[(Uni)] $0 \neq \theta(\pi_1) \simeq \theta(\pi_1')$ if and only if $\pi_1 \simeq \pi_1'$
\end{itemize}

\subsection{} When (Fin)-(Irr)-(Uni) hold, together with the three reverse statements obtained by exchanging the roles of $H_1$ and $H_2$, we say that the $R$-modular local theta correspondence holds or is valid. In this case, the symbol $\theta$ is used in both ways and defines a bijection between subsets of representations that are contributing \textit{i.e.}
$$\{ \pi_1 \in \textup{Irr}_R^{\textup{gen}} (\widehat{H}_{1,S}) \ | \ \Theta(\pi_1) \neq 0\} \overset{\theta}{\simeq} \{ \pi_2 \in \textup{Irr}_R^{\textup{gen}} (\widehat{H}_{2,S}) \ | \ \Theta(\pi_2) \neq 0\}$$
where $\theta(\pi_1) = \pi_2$ if and only if $\omega_{\psi,S} \twoheadrightarrow \pi_1 \otimes_R \pi_2$.

\subsection{} If $(H_1,H_2)$ is a dual pair of type II, \textit{i.e.} if $H_1$ and $H_2$ are general linear groups over a division algebra, then the $R$-modular theta correspondence holds as long as the characteristic $\ell$ of $R$ does not divide the pro-orders of $H_1$ and $H_2$, thanks to the thesis work \cite{minguez_thesis}. When $\ell$ divides the pro-order of $H_1$ or $H_2$, the statement (Irr) does not hold as one can exhibit some $\pi_1$ by \cite[Sec. 4.5.2]{minguez_thesis} such that $\Theta(\pi_1)$ is semisimple of length $2$. If $(H_1,H_2)$ is a dual pair of type I, \textit{i.e.} if $H_1$ and $H_2$ are isometry groups, we refer to \cite[Th A]{trias_modular_theta} for the $R$-modular theta correspondence for non-quaternionic dual pairs if $\ell$ is large enough compared to $H_1$ and $H_2$, but the bound on $\ell$ is not explicit, and to \cite[Th B]{trias_modular_theta} for a counter-example to (Irr) when $\ell$ is not large.

\section{Supercuspidal blocks in banal characteristic}

\subsection{} Let $G$ be a reductive group over $F$ with compact centre. Let $\ell$ be a prime number which does not divide the pro-order of $G$. We write $\ell \nmid |G|$. Let $\mathfrak{r}_\ell : \overline{\mathbb{Z}_\ell} \to \overline{\mathbb{F}_\ell}$.

Let $\mu$ be a Haar measure of $G$ with values in $\overline{\mathbb{Q}_\ell}$. We assume $\mu$ is normalised on some open pro-$p$-subgroup, so that the volume of any open pro-$p$-subgroups subgroup belongs to $p^\mathbb{Z}$. In particular $\mu$ is also a measure with values in $\overline{\mathbb{Z}_\ell}$. Let $\mathfrak{r}_\ell(\mu)$ be the Haar measure with values in $\overline{\mathbb{F}_\ell}$ obtained from $\mu$. The Hecke algebras $\mathcal{H}_{\overline{\mathbb{Q}_\ell}}(G)$ and $\mathcal{H}_{\overline{\mathbb{Z}_\ell}}(G)$ are associated to the measure $\mu$, whereas $\mathcal{H}_{\overline{\mathbb{F}_\ell}}(G)$ is associated to $\mathfrak{r}_\ell(\mu)$. This choice of measures makes the convolution products compatible \textit{i.e.} we have morphisms of algebras
$$\mathcal{H}_{\overline{\mathbb{Z}_\ell}}(G) \hookrightarrow \mathcal{H}_{\overline{\mathbb{Q}_\ell}}(G) \textup{ and } \mathcal{H}_{\overline{\mathbb{Z}_\ell}}(G) \twoheadrightarrow \mathcal{H}_{\overline{\mathbb{F}_\ell}}(G).$$
We use the generic notation $\mathcal{H}_{\mathcal{R}}(G)$ for these algebras where $\mathcal{R}$ is one of $\overline{\mathbb{Q}_\ell}$, $\overline{\mathbb{Z}_\ell}$ or $\overline{\mathbb{F}_\ell}$. 

\subsection{} The centre $\mathfrak{z}_{\mathcal{R}}(G)$  of the category $\textup{Rep}_{\mathcal{R}}(G)$, called the Bernstein centre, is by definition the ring of endomorphisms of the identity functor. This ring is commutative and we can see $z \in \mathfrak{z}_{\mathfrak{R}}(G)$ as a collection $(z_V)_{V \in \textup{Rep}_{\mathcal{R}}(G)}$ of $G$-equivariant endomorphisms such that for all $f \in \textup{Hom}_{\mathcal{R}[G]}(V,W)$, we have $z_W \circ f = f \circ z_V$. The natural action of $\mathfrak{z}_{\mathcal{R}}$ on the regular representation $C_c^\infty(G)$ is faithful \textit{i.e.} $z =(z_V)_V \mapsto z_{C_c^\infty(G,\mathcal{R})}$ is injective. We can even upgrade the latter into a bijection $\mathfrak{z}_{\mathcal{R}}(G) \simeq \textup{End}_{\mathcal{R}[G \times G]}(C_c^\infty(G,\mathcal{R}))$.

We define the functor of scalar extension to $\overline{\mathbb{Q}_\ell}$ as
$$- \otimes_{\overline{\mathbb{Z}_\ell}} \overline{\mathbb{Q}_\ell} : V \in \textup{Rep}_{\overline{\mathbb{Z}_\ell}}(G) \mapsto V \otimes_{\overline{\mathbb{Z}_\ell}} \overline{\mathbb{Q}_\ell} \in \textup{Rep}_{\overline{\mathbb{Z}_\ell}}(G)$$
We also define the reduction modulo $\ell$ functor for representations with coefficients in $\overline{\mathbb{Z}_\ell}$ thanks to $\mathfrak{r}_\ell : \overline{\mathbb{Z}_\ell} \to \overline{\mathbb{F}_\ell}$ and we still denote it by
$$\mathfrak{r}_\ell : V \in \textup{Rep}_{\overline{\mathbb{Z}_\ell}}(G) \mapsto V/\ell V = V \otimes_{\overline{\mathbb{Z}_\ell}} \overline{\mathbb{F}_\ell} \in \textup{Rep}_{\overline{\mathbb{F}_\ell}}(G).$$

\begin{rem} This functor is not to be confused with the usual reduction modulo $\ell$ map, commonly denoted by $r_\ell$ in the literature \cite[II.5.11.b]{vig}, \cite[Sec 4.2]{dhkm_conj}. \end{rem}

Note that we have natural morphisms between centres
$$\mathfrak{z}_{\overline{\mathbb{Z}_\ell}}(G) \hookrightarrow  \mathfrak{z}_{\overline{\mathbb{Q}_\ell}}(G) \textup{ and }\mathfrak{z}_{\overline{\mathbb{Z}_\ell}}(G) \to \mathfrak{z}_{\overline{\mathbb{F}_\ell}}(G),$$
induced by the previous two functors. These morphisms can be easily understood via the regular representation. Indeed, let $\varphi$ be a $(G \times G)$-equivariant endomorphism of $C_c^\infty(G,\overline{\mathbb{Z}_\ell})$. Because $C_c^\infty(G,\overline{\mathbb{Q}_\ell}) \otimes_{\overline{\mathbb{Z}_\ell}} \overline{\mathbb{Q}_\ell} \cong C_c^\infty(G,\overline{\mathbb{Q}_\ell})$ and $\varphi \otimes_{\overline{\mathbb{Z}_\ell}} \overline{\mathbb{Q}_\ell}$ is $(G \times G)$-equivariant, this gives the first ring morphism. It is injective because $\varphi \mapsto \varphi \otimes_{\overline{\mathbb{Z}_\ell}} \overline{\mathbb{Q}_\ell}$ is injective. Similarly $C_c^\infty(G,\overline{\mathbb{Q}_\ell}) \otimes_{\overline{\mathbb{Z}_\ell}} \overline{\mathbb{F}_\ell} \cong C_c^\infty(G,\overline{\mathbb{F}_\ell})$ induces the second ring morphism. However, there is \textit{a priori} no formal reason that would guarantee it is surjective.

\subsection{}  Let $S$ be a subset of $\textup{Irr}_{\mathcal{R}}(G)$ and ${}^c S$ be its complement. Denote by $\textup{Rep}_{\mathcal{R}}^S(G)$ the full subcategory of $\textup{Rep}_{\mathcal{R}}(G)$ whose objects have all their irreducible subquotients in $S$. We say a subset $S$ of $\textup{Irr}_{\mathcal{R}}(G)$ decomposes $\textup{Rep}_{\mathcal{R}}(G)$ if there is a product of categories
$$\textup{Rep}_{\mathcal{R}}(G) = \textup{Rep}_{\mathcal{R}}^S(G) \times \textup{Rep}_{\mathcal{R}}^{{}^c S}(G).$$
In this situation, there exists a (unique) central idempotent $e_S \in \mathfrak{z}_{\mathcal{R}}(G)$ which gives the previous decomposition \textit{i.e.} such that
$$e_S \textup{Rep}_{\mathcal{R}}(G)= \textup{Rep}_{\mathcal{R}}^S(G) \textup{ and } (1-e_S)\textup{Rep}_{\mathcal{R}}(G) = \textup{Rep}_{\mathcal{R}}^{{}^c S}(G).$$
Conversely, any central idempotent $e$ of the Bernstein centre induces a decomposition of the category $\textup{Rep}_{\mathcal{R}}(G)$. By definition, such a decomposition induces a partition in two sets of $\textup{Irr}_{\mathcal{R}}(G)$. We say a central idempotent $e$ is primitive if the category $e \textup{Rep}_{\mathcal{R}}(G)$ is indecomposable. This is equivalent to saying that $e$ can't be written as a sum of two non-zero central idempotents. We say a non-empty subset $S$ of $\textup{Irr}_{\mathcal{R}}(G)$ defines a block in $\textup{Rep}_{\mathcal{R}}(G)$ if $S$ decomposes $\textup{Rep}_{\mathcal{R}}(G)$ and there is no non-empty proper subsets of $S$ decomposing $\textup{Rep}_{\mathcal{R}}(G)$. Finally, the associated central idempotent $e_S$ is primitive if and only if $S$ defines a block.

\subsection{} Let $\Pi \in \textup{Rep}_{\overline{\mathbb{Q}_\ell}}(G)$ be an irreducible cuspidal representation. Since the centre of $G$ is compact and $\ell \nmid |G|$, it is a projective and injective object in the category $\textup{Rep}_{\overline{\mathbb{Q}_\ell}}(G)$. This implies that $\{ \Pi \}$ decomposes $\textup{Rep}_{\overline{\mathbb{Q}_\ell}}(G)$. By Schur's lemma and Morita equivalence
$$\textup{Rep}_{\overline{\mathbb{Q}_\ell}}^{\{ \Pi \}}(G) \simeq \overline{\mathbb{Q}_\ell}-\textup{mod}$$
\textit{i.e.} all representations are $\Pi$-isotypic and only the multiplicity of $\Pi$ matters. We denote by $e_\Pi$ the associated primitive central idempotent.

\subsection{} \label{sec:reduction_mod_ell_of_cuspidal} The representation $\Pi$ is integral \cite[II.4.12]{dhkm_conj} \textit{i.e.} $\Pi$ contains a free $\overline{\mathbb{Z}_\ell}$-lattice $L$ which is stable under the action of $G$. According to \cite[Prop 4.15]{dhkm_conj}, for any stable ${\overline{\mathbb{Z}_\ell}}$-lattice $L$ in $\Pi$, the representation $\pi = \mathfrak{r}_\ell(L)$ is irreducible and cuspidal. The isomorphism class of $\pi$ does not depend on the choice of $L$ as a consequence of the Brauer-Nesbitt principle. In particular, if we consider $\Pi$ as a representation with coefficients in ${\overline{\mathbb{Z}_\ell}}$, this means all its irreducible subquotients are isomorphic to $\pi$. Once again, as the centre of $G$ is compact and $\ell \nmid |G|$, the representation $\pi$ is a projective and injective object in $\textup{Rep}_{\overline{\mathbb{F}_\ell}}(G)$ and the singleton $\{ \pi \}$ decomposes $\textup{Rep}_{\overline{\mathbb{F}_\ell}}(G)$. We let $e_\pi$ be the associated primitive central idempotent.

\subsection{} We want to show $L$ defines a block in $\textup{Rep}_{\overline{\mathbb{Z}_\ell}}(G)$. We will use a compatibility for the formal degrees of $\Pi$ and $\pi$ to show it is projective, which gives an alternative proof of \cite[Prop 4.17]{dhkm_conj}.

Let $R$ be $\overline{\mathbb{Q}_\ell}$ or $\overline{\mathbb{F}_\ell}$. Let $V \in \textup{Rep}_R(G)$ be an irreducible compact representation \textit{i.e.}
$$\begin{array}{ccc} 
V \otimes_R V^\vee & \hookrightarrow & C_c^\infty(G,R) \\
v \otimes_R v^\vee & \mapsto & g \mapsto v^\vee(g \cdot v)
\end{array}.$$
In particular $\Pi$ and $\pi$ are compact representations. Let $\mu$ be a Haar measure with values in $R$. This endows $C_c^\infty(G,R)$ with a structure of algebra via convolution and we have a natural map $C_c^\infty(G,R) \to \textup{End}_R(V)$ given by the action on $V$. The latter map depends on $\mu$. We say $V$ has a formal degree if there exists a Haar measure $\mu^V$ such that the composition $V \otimes_R V^\vee \to C_c^\infty(G,R) \to \textup{End}_R(V)$ is the canonical map $v \otimes_R v^\vee \mapsto (s \mapsto v^\vee(s) \ v)$. If the formal degree exists, it is of course unique.

All projective representations have a formal degree. We now show a compatibility between the formal degrees $\mu^\Pi$ and $\mu^\pi$.

\begin{lem} The formal degree of $\Pi$ comes from a Haar measure with values in $\overline{\mathbb{Z}_\ell}$ and reduces modulo $\ell$ to the formal degree of $\pi$. In particular $L \in \textup{Rep}_{\overline{Z}_\ell}(G)$ is projective. \end{lem}

\begin{proof} Let $\mathcal{B}=(v_i)_{i \in I}$ be a basis of $L$ over $\overline{\mathbb{Z}_\ell}$. We denote by $(v_i^\vee)_{i \in I}$ the dual basis of $\mathcal{B}$ in $\Pi^\vee$. Then the $\overline{\mathbb{Z}_\ell}$-lattice $L^\vee$ generated by this dual basis is a stable $\overline{\mathbb{Z}_\ell}$-lattice in $\Pi^\vee$. By choosing a Haar measure $\mu$ with values in $\overline{\mathbb{Z}_\ell}$, the coefficients of $L$ acts via
$$L \otimes_{\overline{\mathbb{Z}_\ell}} L^\vee \to C_c^\infty(G,\overline{\mathbb{Z}_\ell}) \to \textup{End}_{\overline{\mathbb{Z}_\ell}}(L)$$
where only the last morphism depends on $\mu$.

Moreover the functors $- \otimes_{\overline{\mathbb{Z}_\ell}} \overline{\mathbb{Q}_\ell}$ and $\mathfrak{r}_\ell$ induce $\pi \otimes_{\overline{\mathbb{F}_\ell}} \pi^\vee \to C_c^\infty(G,\overline{\mathbb{F}_\ell}) \to \textup{End}_{\overline{\mathbb{F}_\ell}}(\pi)$ and $\Pi \otimes_{\overline{\mathbb{Q}_\ell}} \Pi^\vee \to C_c^\infty(G,\overline{\mathbb{Q}_\ell}) \to \textup{End}_{\overline{\mathbb{Q}_\ell}}(\Pi)$. By existence of the formal degree, these maps are scalar multiples of the canonical map. Let $a_\Pi \in \overline{\mathbb{Q}_\ell}$ and $a_\pi \in \overline{\mathbb{F}_\ell}$ be those scalars.

We easily see that $a_\Pi \neq 0$, otherwise $\mu$ must be zero. Moreover, if we assume that $\mathfrak{r}_\ell(\mu) \neq 0$, then we see that $a_\pi \neq 0$ and $a_\Pi \in \overline{\mathbb{Z}_\ell}$ and $\mathfrak{r}_\ell(a_\Pi) = a_\pi$. As a result $a_\Pi \in \overline{\mathbb{Z}_\ell}^\times$.
Therefore the measure $a_\Pi^{-1} \mu$ can be identified with $\mu^\Pi$ and $\mathfrak{r}_\ell(\mu^\Pi)=\mu^\pi$.

As $L$ is admissible, the image of $C_c^\infty(G,\overline{\mathbb{Z}_\ell}) \to \textup{End}_{\overline{\mathbb{Z}_\ell}}(L)$ lies in $\textup{End}_{\overline{\mathbb{Z}_\ell}}^{\textup{fin}}(L)$ where $f \in \textup{End}_{\overline{\mathbb{Z}_\ell}}^{\textup{fin}}(L)$ if there exists $L' \subset L$ of finite rank such that $L = L' \oplus L''$ and $f$ factors through an endomorphism in $\textup{End}_{\overline{\mathbb{Z}_\ell}}(L')$. Furthermore $\textup{End}_{\overline{\mathbb{Z}_\ell}}^{\textup{fin}}(L) \simeq L \otimes_{\overline{\mathbb{Z}_\ell}} L^\vee$ in $\textup{Rep}_{\overline{\mathbb{Z}_\ell}}(G \times G)$. Therefore $L \otimes_{\overline{\mathbb{Z}_\ell}} L^\vee \hookrightarrow C_c^\infty(G,\overline{\mathbb{Z}_\ell})$ admits a retract \textit{i.e.}
$$C_c^\infty(G,\overline{\mathbb{Z}_\ell}) \simeq (L \otimes_{\overline{\mathbb{Z}_\ell}} L^\vee) \oplus V'.$$
But $C_c^\infty(G,\overline{\mathbb{Z}_\ell})$ is projective by \cite[Lem 1.4]{trias_modular_theta}, so $L$ is projective as well. \end{proof}

\begin{prop} The idempotent $e_\Pi$ belongs to $\mathfrak{z}_{\overline{\mathbb{Z}_\ell}}(G)$ and $\mathfrak{r}_\ell(e_\Pi) = e_\pi$. \end{prop}

\begin{proof} From the previous decomposition
$$C_c^\infty(G,\overline{\mathbb{Z}_\ell}) \simeq (L \otimes_{\overline{\mathbb{Z}_\ell}} L^\vee) \oplus V',$$
gives by scalar extension to $\overline{\mathbb{Q}_\ell}$ the decomposition
$$C_c^\infty(G,\overline{\mathbb{Q}_\ell}) = e_\Pi C_c^\infty(G,\overline{\mathbb{Q}_\ell}) \oplus (1-e_\Pi)C_c^\infty(G,\overline{\mathbb{Q}_\ell})$$
because $e_\Pi C_c^\infty(G,\overline{\mathbb{Q}_\ell}) \simeq \Pi \otimes_{\overline{\mathbb{Q}_\ell}} \Pi^\vee$. In particular $\textup{Hom}_{\overline{\mathbb{Z}_\ell}[G]}(L,V') = 0$, so no subquotient of $V'$ can be isomorphic to $\pi$, otherwise $\textup{Hom}_{\overline{\mathbb{Z}_\ell}[G]}(L,V') \neq 0$ as $L$ is projective. Therefore there exists $e \in \mathfrak{z}_{\overline{\mathbb{Z}_\ell}}(G)$ such that $e C_c^\infty(G,\overline{\mathbb{Z}_\ell}) \simeq L \otimes_{\overline{\mathbb{Z}_\ell}} L^\vee$ and $e = e_\Pi$ in $\mathfrak{z}_{\overline{\mathbb{Q}_\ell}}(G)$. We have $\mathfrak{r}_\ell(e_\Pi) = e_\pi$ because $\pi$ is the only irreducible subqutotient of $L$. \end{proof}

Then $e_\Pi \textup{Rep}_{\overline{\mathbb{Z}_\ell}}(G) = \textup{Rep}_{\overline{\mathbb{Z}_\ell}}^{\{\pi\}}(G)$ is clearly a block and we also have $\mathfrak{r}_\ell(e_\Pi V) = e_\pi \mathfrak{r}_\ell(V)$ for all $V \in \textup{Rep}_{\overline{\mathbb{Z}_\ell}}(G)$.

\section{Congruences of supercuspidal theta lifts in banal characteristic} \label{sec:congruences_for_cuspidal_theta_lifts}

Let $(H_1,H_2)$ be a dual pair of type I in $\textup{Sp}(W)$ over a non-archimedean local field $F$. Let $\psi$ be a non-trivial character of $F$ with values in $\overline{\mathbb{Z}_\ell}$. In particular we can consider $\psi$ as being valued in $\overline{\mathbb{Q}_\ell}$, and its reduction modulo $\ell$ defines a non-trivial character with values in $\overline{\mathbb{F}_\ell}$, still denoted $\psi$. We allow this abuse of notations as the context should be clear. Let $\mathcal{R}$ be any of $\overline{\mathbb{Q}_\ell}$, $\overline{\mathbb{Z}_\ell}$ or $\overline{\mathbb{F}_\ell}$.

\subsection{} Let $X$ be a Lagrangian in $W$. Let $V_X^{\mathcal{R}}$ be the Schr\"odinger model of the Heisenberg representation with coefficients in $\mathcal{R}$ associated to $\psi$ and $X$. When $\mathcal{R}=\overline{\mathbb{Z}_\ell}$, we simply recall from \cite{trias_theta2} that is the subspace of functions valued in $\overline{\mathbb{Z}_\ell}$. We have equivariant morphisms for the action of the Heisenberg group
$$V_X^{\overline{\mathbb{Z}_\ell}} \hookrightarrow V_X^{\overline{\mathbb{Q}_\ell}} \textup{ and } V_X^{\overline{\mathbb{Z}_\ell}} \twoheadrightarrow V_X^{\overline{\mathbb{F}_\ell}}.$$
The model of the Weil representation over $\mathcal{R}$ associated to $\psi$ and $X$ is
$$(\omega_{\psi,V_X^{\mathcal{R}}},V_X^{\mathcal{R}}) \in \textup{Rep}_R(\widehat{\textup{Sp}}(W)),$$
where $\widehat{\textup{Sp}}(W)$ is the metaplectic group. The previous morphisms are equivariant \textit{i.e.}
$$\omega_{\psi,V_X^{\overline{\mathbb{Z}_\ell}}} \hookrightarrow \omega_{\psi,V_X^{\overline{\mathbb{Q}_\ell}}} \textup{ and } \omega_{\psi,V_X^{\overline{\mathbb{Z}_\ell}}} \twoheadrightarrow \omega_{\psi,V_X^{\overline{\mathbb{F}_\ell}}}.$$
We denote by $\widehat{H}_1$ and $\widehat{H}_2$ the inverse images of $H_1$ and $H_2$ in $\widehat{\textup{Sp}}(W)$. Therefore
$$\omega_{\psi,V_X^{\mathcal{R}}} \in \textup{Rep}_{\mathcal{R}}( \widehat{H}_1 \times \widehat{H}_2).$$

\subsection{} We now suppose that $\ell$ does not divide the pro-order of $\widehat{H}_1$. We also assume that $\widehat{H}_1$ is split over $H_1$ \textit{i.e.} $\widehat{H}_1 \simeq H_1 \times \{ \pm 1 \}$. Then we have the following equivalence of categories $\textup{Rep}_{\mathcal{R}}^{\textup{gen}}(\widehat{H}_1) \simeq \textup{Rep}_{\mathcal{R}}(H_1)$.
These categories share the same properties in the sense that a representation in $\textup{Rep}_{\mathcal{R}}(\widehat{H}_1)$ is projective, resp. injective or cuspidal or integral, if and only its image in $\textup{Rep}_{\mathcal{R}}(H_1)$ is. We suppose $\hat{H}_2$ is split as well.

Let $\Pi_1 \in \textup{Rep}_{\overline{\mathbb{Q}_\ell}}(H_1)$ be irreducible and cuspidal. Let $\pi_1 \in \textup{Rep}_{\overline{\mathbb{F}_\ell}}(H_1)$ be the irreducible cuspidal representation appearing in Section \ref{sec:reduction_mod_ell_of_cuspidal}. A famous result in the complex setting \cite[Chap. 3, IV.4 Th. 1) a)]{mvw} ensures that $\Theta(\Pi_1)$ is irreducible when it is non-zero. This result is also valid over $\overline{\mathbb{Q}_\ell}$ as $\overline{\mathbb{Q}_\ell} \simeq \mathbb{C}$. If $\Theta(\Pi_1)$ is integral, it admits a stable $\overline{\mathbb{F}_\ell}$-lattice $L$ and the Brauer-Nesbitt principle then guarantees the semisimplification $r_\ell(\Pi_1)$ of $\mathfrak{r}_\ell(L)$ has finite length and is independent of the choice of $L$.

\begin{prop} \label{prop:reduction_theta_Brauer_Nesbitt} We recall that $\ell$ does not divide the pro-order of $\hat{H}_1$ and we assume $\hat{H}_1$ and $\hat{H}_2$ are split. Then the representation $\Theta(\Pi_1)$ is integral and the semisimplification of $\Theta(\pi_1)$ is $r_\ell(\Theta(\Pi_1))$. In particular $\Theta(\pi_1)$ has finite length. \end{prop}

\begin{proof} Because the category $\textup{Rep}_{\overline{\mathbb{Q}_\ell}}^{\{\Pi_1\}}(H_1)$ is semisimple
$${\big(\omega_{\psi,V_X^{\overline{\mathbb{Q}_\ell}}} \big)}_{\Pi_1} \simeq e_{\Pi_1} \omega_{\psi,V_X^{\overline{\mathbb{Q}_\ell}}} \simeq \Pi_1 \otimes_{\overline{\mathbb{Q}_\ell}} \Theta(\Pi_1) \textup{ in } \textup{Rep}_{\overline{\mathbb{Q}_\ell}}(H_1 \times H_2).$$
We consider $e_{\Pi_1} \omega_{\psi,V_X^{\overline{\mathbb{Z}_\ell}}} \subseteq e_{\Pi_1} \omega_{\psi,V_X^{\overline{\mathbb{Q}_\ell}}}$ in $\textup{Rep}_{\overline{\mathbb{Z}_\ell}}(H_1 \times H_2)$. In addition
$$\mathfrak{r}_\ell(e_{\Pi_1} \omega_{\psi,V_X^{\overline{\mathbb{Z}_\ell}}}) = e_{\pi_1} \mathfrak{r}_\ell(\omega_{\psi,V_X^{\overline{\mathbb{Q}_\ell}}}) = e_{\pi_1} \omega_{\psi,V_X^{\overline{\mathbb{F}_\ell}}} \simeq \pi_1 \otimes_{\overline{\mathbb{F}_\ell}} \Theta(\pi_1).$$
To finish the proof, we need to show that $e_{\Pi_1} \omega_{\psi,V_X^{\overline{\mathbb{Z}_\ell}}} \in \textup{Rep}_{\overline{\mathbb{Z}_\ell}}(H_1 \times H_2)$ is a stable $\overline{\mathbb{Z}_\ell}$-lattice in the irreducible representation
$$e_{\Pi_1} \omega_{\psi,V_X^{\overline{\mathbb{Q}_\ell}}} \simeq \Pi_1 \otimes_{\overline{\mathbb{Q}_\ell}} \Theta(\Pi_1)  \in \textup{Rep}_{\overline{\mathbb{Q}_\ell}}(H_1 \times H_2).$$
However, we have to replace $\overline{\mathbb{Z}_\ell}$ by a finite extension of $\mathbb{Z}_\ell$ in our argument. First of all, by a forthcoming work \cite{trias_rational_weil}, there exists a finite extension $E$ of $\mathbb{Q}_\ell$ and $\omega \in \textup{Rep}_{\mathcal{O}_E}(H_1 \times H_2)$ such that $\omega \otimes_{\mathcal{O}_E} \mathcal{R} \simeq \omega_{\psi,V_X^{\mathcal{R}}}$. Moreover, up to enlarging $E$ again, we can assume $\Pi_1$ and $\pi_1$ are realised over $E$ and $k_E$. By abuse of notations, we still denote these representations by $\Pi_1$ and $\pi_1$. We also write $e_{\Pi_1} \in \mathfrak{z}_{\mathcal{O}_E}(G)$ and $e_{\pi_1} \in \mathfrak{z}_{k_E}(G)$.

The representation $\omega$ does not contain any $E$-lines. As $\omega = e_{\Pi_1} \omega \oplus (1-e_{\Pi_1}) \omega$, then $e_{\Pi_1} \omega$ does not contain an $E$-line. Furthermore $e_{\Pi_1} \omega \subseteq e_{\Pi_1} (\omega \otimes_{\mathcal{O}_E} E) \simeq \Pi_1 \otimes_E \Theta(\Pi_1)$ where the latter is an irreducible admissible representation (assuming $\Theta(\Pi_1) \neq 0$). Since $\mathcal{O}_E$ is local principal complete and $\textup{dim}_E(\Pi_1 \otimes_E \Theta(\Pi_1))$ is countable, we deduce that $e_{\Pi_1} \omega$ is a stable $\mathcal{O}_E$-lattice by \cite[I.9.2]{vig}. Therefore $e_{\Pi_1} \omega \otimes_{\mathcal{O}_E} \overline{\mathbb{Z}_\ell}$ is a stable $\overline{\mathbb{Z}_\ell}$-lattice. In particular $\Theta(\Pi_1)$ is integral.

By \cite[II.5.11]{vig} $\mathfrak{r}_\ell(e_{\Pi_1} \omega_{\psi,V_X^{\overline{\mathbb{Z}_\ell}}})=\pi_1 \otimes_{\overline{\mathbb{F}_\ell}} \Theta(\pi_1)$ has finite length, so does $\Theta(\pi_1)$. \end{proof}

We can improve the previous result in the so-called banal case \textit{i.e.} when $\ell$ does not divide any of the pro-orders of $H_1$ and $H_2$. The condition $\Theta(\Pi_1)$ is irreducible cuspidal recovers the famous case of the first occurrence index.

\begin{theo} Assume $\ell$ does not divide the pro-orders of $H_1$ and $H_2$, and assume $\hat{H}_1$ and $\hat{H}_2$ are split. Suppose $\Theta(\Pi_1)$ is an irreducible cuspidal representation. Then $\Theta(\pi_1)$ is an irreducible cuspidal representation. \end{theo}

\begin{proof} By \cite[Prop 4.15]{dhkm_conj}, the representation $r_\ell(\Theta(\Pi_1))$ is irreducible and cuspidal. It is also the semisimplified of $\Theta(\pi_1)$, so
$\Theta(\pi_1)$ is irreducible and cuspidal. \end{proof}

\begin{rem} Though it is probably a hard question, it would be nice to have a more precise description of the stable $\overline{\mathbb{Z}_\ell}$-lattice 
$$e_{\Pi_1} \omega_{\psi,V_X^{\overline{\mathbb{Z}_\ell}}} \subseteq e_{\Pi_1} \omega_{\psi,V_X^{\overline{\mathbb{Q}_\ell}}} \simeq \Pi_1 \otimes_{\overline{\mathbb{Q}_\ell}} \Theta(\Pi_1)$$ 
to describe $\Theta(\pi_1)$ since 
$$\mathfrak{r}_\ell(e_{\Pi_1} \omega_{\psi,V_X^{W(k)}}^{W(k)}) = \pi_1 \otimes_{\overline{\mathbb{F}_\ell}} \Theta(\pi_1).$$
At the moment, the only information we have control on is about $r_\ell(\Theta(\Pi_1))$, which is the semisimplification of $\Theta(\pi_1)$. So it does not tell us much about $\theta(\pi_1)$, except if $r_\ell(\Theta(\Pi_1))$ has length at most $1$. Also, even when $\ell$ is banal with respect to $H_1$ and $H_2$, it is not true any irreducible integral representation has irreducible reduction modulo $\ell$. Therefore it is not possible to generalise our arguments beyond the cuspidal case. \end{rem}

\bibliographystyle{alpha}
\bibliography{lesrefer}

\end{document}